\documentclass[11pt,a4paper]{amsart}
\usepackage{amssymb,amsthm}
\usepackage{tikz,tikz-cd}
\usepackage[smalltableaux,centertableaux]{ytableau}
\usepackage[pagebackref]{hyperref} 
\hypersetup{colorlinks,linkcolor=blue,urlcolor=blue,citecolor=blue,%
  linktocpage}  
\usepackage[alphabetic,abbrev,nobysame]{amsrefs} 

\newcommand{\X}{\mathbb{X}} 
\newcommand{\C}{\mathbb{C}} 
\newcommand{\N}{\mathbb{N}} 
\newcommand{\Z}{\mathbb{Z}} 
\newcommand{\Q}{\mathbb{Q}} 
\newcommand{\A}{\mathcal{A}} 
\newcommand{\End}{\operatorname{End}} 
\newcommand{\Hom}{\operatorname{Hom}} 
\newcommand{\fgl}{\mathfrak{gl}} 
\newcommand{\fsl}{\mathfrak{sl}} 
\newcommand{\ov}{\overline} 
\newcommand{\UU}{\mathbf{U}} 
\newcommand{\Sym}{\mathfrak{S}} 
\newcommand{\ep}{\varepsilon}
\newcommand{\Specht}{\mathsf{S}} 
\newcommand{\HH}{\mathbf{H}} 
\newcommand{\bil}[2]{\left\langle #1, #2 \right\rangle} 
\newcommand{\bigbil}[2]{\big\langle #1, #2 \big\rangle} 
\newcommand{\smbil}[2]{\langle #1, #2 \rangle} 
\newcommand{\tK}{\widetilde{K}} 
\newcommand{\cc}{\mathfrak{c}}  
\newcommand{\Walk}{\mathsf{Walk}} 
\newcommand{\aaa}{\mathbf{a}} 
\newcommand{\ch}{\operatorname{ch}} 
\newcommand{\mult}{\operatorname{d}} 
\newcommand{\st}{{\operatorname{st}}} 
\newcommand{\sst}{{\operatorname{sst}}} 
\newcommand{\Mod}[1]{#1\operatorname{-mod}} 
\newcommand{\sqbinom}[2]{\left[\begin{matrix}#1\\#2\end{matrix}\right]}
\newcommand{\tsp}{\mathsf{T}}  
\newcommand{\fU}{\mathfrak{U}} 

\newcommand{\Part}[1]{
 \foreach \x [count=\s from 1] in {#1}{
 	{\ifnum\s=1
		\draw (0,\s-1)--(\x,\s-1); 
		\fi}
   \draw (0,\s) to (\x,\s);
   \foreach \y in {0, ..., \x} {\draw (\y,\s)--(\y,\s-1);}
 }}

\def\UNIT{.18} \newcommand{\PART}[1]{
\begin{tikzpicture}[xscale=\UNIT, yscale=-\UNIT] 
	\Part{#1}
\end{tikzpicture}
}

\swapnumbers 
\newtheorem{thm}{Theorem}[section]
\newtheorem*{thm*}{Theorem}
\newtheorem{lem}[thm]{Lemma}
\newtheorem*{lem*}{Lemma}
\newtheorem{prop}[thm]{Proposition}
\newtheorem*{prop*}{Proposition}
\newtheorem{cor}[thm]{Corollary}
\newtheorem*{cor*}{Corollary}

\newtheorem*{conj*}{Conjecture}

\theoremstyle{definition}
\newtheorem{defn}[thm]{Definition}
\newtheorem*{defn*}{Definition}
\newtheorem{example}[thm]{Example}
\newtheorem*{example*}{Example}
\newtheorem{rmk}[thm]{Remark}
\newtheorem*{rmk*}{Remark}

\newtheorem*{que*}{Question}

\allowdisplaybreaks

\title[Orthogonal decomposition of tensor space]%
      {Lowering operators, orthogonal decomposition of tensor space, 
      and quantized Schur--Weyl duality}

\author{Stephen Doty}
  \address{Department of Mathematics and Statistics, Loyola University
    Chicago, Chicago, IL 60660 USA}
  \email{doty@math.luc.edu}

\author{Anthony Giaquinto}
  \address{Department of Mathematics and Statistics,
  Loyola University Chicago, Chicago, IL 60660 USA}
   \email{tonyg@math.luc.edu}

\author{Stuart Martin}
  \address{DPMMS, Centre for Mathematical Sciences, Wilberforce Road,
  Cambridge, CB3 0WB, UK}
  \email{sm137@cam.ac.uk}


\thanks{The first author acknowledges support of the National 
Science Foundation under Grant No.~DMS-1929284 while 
in residence at the Institute for Computational and 
Experimental Research in Mathematics in Providence, RI, 
during the Categorification and Computation in Algebraic 
Combinatorics program. 
All the the authors thank the referees of an earlier version for 
suggesting possible connections with lowering operators.}

\begin{document}
\begin{abstract}\noindent
For $q$ generic, Jimbo showed that $q$-tensor space $V_q^{\otimes r}$ 
(where $V_q$ is the $n$-dimensional vector representation) 
satisfies Schur--Weyl duality with respect to the
commuting actions of the quantized enveloping algebra $\UU_q(\fgl_n)$ 
and the Iwahori--Hecke algebra $\HH_q(\Sym_r)$, with the latter 
action derived from the $R$-matrix. 
In the limit as $q \to 1$, one recovers classical
Schur--Weyl duality.


Using a recursive construction of certain linear combinations $\Psi_j$ 
of Coxeter monomials in the negative part of $\UU_q(\fgl_n)$, 
we give a combinatorial realization of the corresponding isotypic 
semisimple decomposition of $V_q^{\otimes r}$, 
indexed by paths in the Bratteli diagram. 
This extends earlier work (\emph{Journal of Algebra}
2024) of the first two authors for the case $n =2$. 
Our construction works over any field containing a non-zero element 
$q$ which is not a root of unity.

The element $\Psi_j$ depends on a weight $\lambda$ and 
is the ``evaluation at $\lambda$'' of a certain $q$-lowering operator
$\ov{\Psi}_j$ satisfying a similar recursion, up to renormalization.
This simplifies the construction of lowering operators. Both
$\Psi_j$ and $\ov{\Psi}_j$ are independent of a 
choice of root vectors. On the other hand, the $\Psi_j$ 
can be applied to construct root vectors 
(independent of the braid group action) 
as explicit linear combinations of Coxeter monomials.
\end{abstract}
\maketitle

%
\section{Introduction}\label{s:intro}\noindent
Let $\Bbbk$ be a field. Fix $0 \ne q \in \Bbbk$ which is not a root of
unity. Let $\UU_q = \UU_q(\fgl_n)$ be the quantized enveloping
algebra of $\fgl_n$ over $\Bbbk$ at $x = q$; by this we mean the
specialization of Lusztig's divided power $\Z[x,x^{-1}]$-form via $x
\mapsto q$. As $q$ is not a root of unity, $\UU_q$ may also be defined
by generators and relations as in \cite{Jimbo} (slightly modified by
Lusztig).

Throughout this paper, we identify partitions with dominant polynomial
weights.  Let $V_q(\lambda)$ be the (type $\mathbf{1}$) $q$-Weyl
module of highest weight $\lambda$.  Set $V_q = V_q(1)$, the vector
representation of $\UU_q$.  Jimbo observed that the $R$-matrix induces
an action of the Iwahori--Hecke algebra $\HH_q = \HH_q(\Sym_r)$ on
tensor space $V_q^{\otimes r}$, commuting with the $\UU_q$-action,
and that these commuting actions satisfy Schur--Weyl duality (each
action generates the full centralizer of the other), in the generic
case where the ground field is $\C(x)$ and $q=x$.  Taking $q \to 1$, 
this becomes classical Schur--Weyl duality.

It turns out that Jimbo's result holds over any field containing 
some $q\ne 0$ which is not a root of unity. 
As Jimbo never published his proof, we
include a proof (of the more general statement) in Section
\ref{s:SWD}. Our proof relies on the fact that the representations of
$\UU_q$ and $\HH_q$ ``behave the same'' as the representations of
their classical counterparts $U(\fgl_n)$ and $\C[\Sym_r]$ in the non
root of unity case.  Then standard semisimplicity theory yields as a
corollary that tensor space $V_q^{\otimes r}$ admits a
multiplicity-free decomposition
\begin{equation}\label{e:isotypic}
V_q^{\otimes r} \cong \textstyle \bigoplus_{\lambda}
V_q(\lambda) \otimes \Specht_q^\lambda
\end{equation}
as $(\UU_q, \HH_q)$-bimodules, where the sum is over the set of
partitions of $r$ into at most $n$ parts, and where
$\Specht_q^\lambda$ is the (simple) $q$-Specht module indexed by the
partition $\lambda$. The $q$-Specht module $\Specht^\lambda$ has a basis
indexed by the set of standard tableaux of shape $\lambda$, so
\eqref{e:isotypic} says that, as a $\UU_q$-module,
\begin{equation}\label{e:isotypic-2}
V_q^{\otimes r} \cong \textstyle\bigoplus_{\mathsf{T}}
V_q(\text{shape}(\mathsf{T}))
\end{equation}
where $\mathsf{T}$ varies over the set of standard tableaux with $r$
boxes and at most $n$ rows. Instead of indexing the sum in
\eqref{e:isotypic-2} by standard tableaux, we find it more convenient
to use walks on the Bratteli diagram. The equivalence between the two
indexing systems is explained at the end of this introduction.

Let $\{v_i\}_{i=1}^n$ be the standard basis of $V_q$. The basis is
orthonormal with respect to the usual bilinear form, defined by
$\bil{v_i}{v_j} = \delta_{ij}$. The standard basis $\{v_{i_1} \otimes
\cdots \otimes v_{i_r} \mid 1 \le i_j \le n \}$ is orthonormal with
respect to the natural extension of the form to $V_q^{\otimes r}$. The
purpose of this paper is to combinatorially construct a full set
\[
\{ \cc_\pi \mid \pi \in \Walk(r) \}
\]
of pairwise orthogonal highest weight vectors (\emph{maximal vectors}
by another name) in $V_q^{\otimes r}$, indexed by the set of walks in
the Bratteli diagram of length $r$, that realizes the decomposition
in~\eqref{e:isotypic-2}; 
see \cites{Gyoja,Carter,Ram-Wenzl,Brundan:98a,Brundan:98b} for related
results. We summarize the main consequences of our construction.

\begin{thm*}
Suppose that $0 \ne q \in \Bbbk$ is not a root of unity. We write $\pi
\to \lambda$ to mean that a walk $\pi$ terminates at a node $\lambda$
in the Bratteli diagram. Let $\lambda$ be a partition of $r$ into not
more than $n$ parts.  Then:
\begin{enumerate}
\item $\UU_q \cc_\pi \cong V_q(\lambda)$, as $\UU_q$-modules, for any
  $\pi \to \lambda$.
\item $V_q^{\otimes r} \cong \bigoplus_{\pi \in \Walk(r)} \UU_q \cc_\pi$.
\item The $\Bbbk$-span of $\{\cc_\pi \mid \pi \to \lambda\}$ is
  isomorphic to $\Specht_q^\lambda$, as $\HH_q$-modules.
\end{enumerate}
\end{thm*}

Part (a) follows from the universal property of $q$-Weyl modules, and
the fact (Theorem~\ref{t:main}) that the $\cc_\pi$ are highest weight
vectors.  Part (b) follows from the fact that the $\cc_\pi$ are
non-isotropic and pairwise orthogonal
(Corollary~\ref{c:orthogonality}) so we have the correct number of
linearly independent highest weight vectors. Part (c) follows from
part (b) and Schur--Weyl duality (see Corollary~\ref{c:SWD}). This
completes the proof.

To restate the result in another way, we have constructed a disjoint union
\[
\{ \cc_\pi \mid \pi \in \Walk(r) \} = \textstyle
\bigsqcup_\lambda \{ \cc_\pi \mid \pi \to \lambda \}
\]
where $\lambda$ varies over the set of partitions of $r$ into not more
than $n$ parts, with the property that $\{ \cc_\pi \mid \pi \to
\lambda \}$ is an orthogonal basis of $\Specht_q^\lambda$. This also
means that the multiplicity of $\Specht_q^\lambda$ in a semisimple
decomposition of $V_q^{\otimes r}$, as an $\HH_q$-module, is equal to
$\dim_\Bbbk V_q(\lambda)$.

As an application, we obtain a basis for the algebra of invariants for
the restricted action of $\UU_q(\fsl_n)$. Note that we we get
invariants only in tensor degrees $r$ such that $r \equiv 0$ modulo $n$.

\begin{cor*}
$\{ \cc_\pi \mid \pi \to \lambda \text{ and $\lambda = (n^j)$ for some
    $j$} \}$ is a basis of the algebra of invariants $(V_q^{\otimes
    r})^{\UU_q(\fsl_n)}$ for the restricted action of $\UU_q(\fsl_n)$.
\end{cor*}

This holds because each highest weight vector $\cc_\pi$ such that $\pi
\to (n^j)$ generates a one-dimensional $q$-Weyl module isomorphic to
the $j$th power of the the $q$-determinant representation, which
becomes trivial upon restriction.

In principle, tensor space can be decomposed using crystal bases, but
our method is much more elementary. In particular, it produces
explicit formulas for the highest weight vectors, constructed
inductively, and the construction is easy to implement on a digital
computer. We use nothing more than root system combinatorics and
linear algebra.  The method should extend to other types of root
systems.

In our construction, each $\cc_\pi$ is obtained by a composition of
$\Phi$ operators applied to the unit $1$, regarded as a basis for the
zeroth tensor power of $V_q$. The precise sequence of operators in the
composition is determined by the walk $\pi$; see
Definition~\ref{d:Upsilon}. There are $n$ such operators, denoted
$\Phi_1, \dots, \Phi_n$. The operator $\Phi_j$ has the following
property: $\Phi_j(b)$ is a highest weight vector if $b$ is one, and if
the weight of $b$ is such that it is possible to add a node to the
$j$th row of its Young diagram. The resulting Young diagram obtained
by adding that node gives the weight of $\Phi_j(b)$. Our construction
extends the construction of similar operators $\Phi_1$, $\Phi_2$ in
\cite{DG:orthog}, except that in the present paper we have reversed the
order of tensor products in order to simplify certain powers of $q$.

The $\Phi$ operators depend on certain elements $\Psi_j$ in the
negative part of the quantized enveloping algebra. In fact, $\Psi_j$
is a linear combination of \emph{Coxeter monomials}, where a Coxeter
monomial is a product $F_1\cdots F_j$ of negative part generators, up
to reordering. 
When $q = 1$, we show in Section~\ref{s:q=1} that $\Psi_j$ and
Carter's lowering operator $S_{1,j+1}$ defined in \cite{Carter} 
have the same effect on weight vectors,
up to scalar multiple and a twist by the automorphism induced from 
the graph automorphism of the underlying Dynkin diagram. 
The twist is a consequence of our choice of embedding of $\UU_q(\fgl_n)$
into $\UU_q(\fgl_{n+1})$ via a ``shifting'' of the generators. 
Brundan \cites{Brundan:98a, Brundan:98b} defined $q$-analogues 
of Carter's lowering operators which depend on a choice of root vectors; 
our recursive method avoids making any such choice.
The definition of $\Psi_j$ depends on a given weight
$\lambda$, and if all the nodes in $\lambda$ are addable then the
various (shifts of) the $\Psi$'s appearing in the formulas for the
$\Phi_2, \dots, \Phi_n$ define a complete set of negative root vectors
in the negative part $\UU_q^-$; see
Theorem~\ref{t:root-vectors}. These negative root vectors may be of
independent interest.

The paper is organized as follows. In Section \ref{s:pre} we establish
our notation and recall basic definitions and results. Section
\ref{s:SWD} gives the aforementioned proof of Jimbo's Schur--Weyl
duality.  Section \ref{s:Coxeter} introduces the notion of a Coxeter
monomial in the negative part $\UU_q^-$ of $\UU_q$, based on the
notion of Coxeter element \cite{Humphreys} in a Coxeter group. Section
\ref{s:Psi} recursively defines certain elements $\Psi_1, \dots,
\Psi_{n-1}$ in $\UU_q^-$ as linear combinations of Coxeter monomials
and develops their initial properties. Section \ref{s:Phi} defines
operators $\Phi_1, \dots, \Phi_n$ (depending on the $\Psi_j$) having
the property that $\Phi_j(b)$ is a highest weight vector if $b$ is,
for each $j$, and constructs the $\cc_\pi$.  Section \ref{s:orthog}
proves pairwise orthogonality of the $\cc_\pi$. Finally, Section
\ref{s:further} gives some additional properties of the $\Psi_j$
elements and Section~\ref{s:q=1} analyzes the relation to the $q=1$ case.
Section ~\ref{s:q=1} was added after the first version of this paper 
was submitted, in response to suggestions from the referees to 
clarify the connection to lowering operators.

\textbf{Bratteli diagram.}
The Bratteli diagram is a graph constructed inductively as follows.
The vertices are partitions into not more than $n$ parts, with
the empty partition $\emptyset$ being the sole vertex at level
zero. Vertices in level $r$ are the partitions of $r$ into not more
than $n$ parts, ordered by some total ordering compatible with reverse
dominance. Draw an edge between two vertices $\lambda$ and $\mu$ lying
in successive levels if and only if $\lambda$ and $\mu$ differ by
exactly one box, which is an addable node (see Section~\ref{s:pre})
for the lower level partition.  We display the first few levels of the
Bratteli diagram in Figure~\ref{fig:1}. 
A \emph{walk} on the Bratteli diagram is a connected piecewise
linear decreasing path
\[
\pi = ( \emptyset \to \pi^{(1)} \to \cdots \to \pi^{(r-1)} \to
\pi^{(r)} = \lambda )
\]
from the empty vertex $\emptyset$ to some vertex $\lambda$, where
$\pi^{(j)}$ is a vertex at level $j$ for each $j$.

%
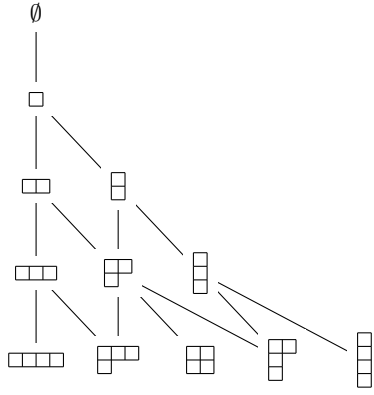
\begin{figure}[ht]
\begin{center}
\begin{tikzpicture}[xscale=3*\UNIT, yscale=-3.2*\UNIT]
  \coordinate (0) at (0,0);
  \foreach \x in {0,...,0}{\coordinate (1\x) at (2*\x,2);}
  \foreach \x in {0,...,1}{\coordinate (2\x) at (2*\x,4);}
  \foreach \x in {0,...,2}{\coordinate (3\x) at (2*\x,6);}
  \foreach \x in {0,...,4}{\coordinate (4\x) at (2*\x,8);}
  
  \draw (0)--(10);

  \draw (10)--(20) (10)--(21);

  \draw (20)--(30) (20)--(31) (21)--(31) (21)--(32); 

  \draw (30)--(40) (30)--(41) (31)--(41) (31)--(42) (31)--(43)
  (32)--(43) (32)--(44);

\def\NULL{\footnotesize$\emptyset$}
\begin{scope}[every node/.style={fill=white}]
  \node at (0) {\NULL};

  \node at (10) {\PART{1}};

  \node at (20) {\PART{2}};
  \node at (21) {\PART{1,1}};

  \node at (30) {\PART{3}};
  \node at (31) {\PART{2,1}};
  \node at (32) {\PART{1,1,1}};

  \node at (40) {\PART{4}};
  \node at (41) {\PART{3,1}};
  \node at (42) {\PART{2,2}};
  \node at (43) {\PART{2,1,1}};
  \node at (44) {\PART{1,1,1,1}};

\end{scope}
\end{tikzpicture}
\end{center}
\caption{Bratteli diagram up to level $4$}\label{Bratteli} \label{fig:1}
\end{figure}
%

If $\pi$ is a walk to $\lambda$, we construct a sequence of standard
tableaux, as follows. The vertices $\pi^{(j)}$ and $\pi^{(j-1)}$
differ by an addable node.  Enter $j$ into that addable node, for each
$j$. In this way we obtain a standard tableau for each vertex in the
walk. We denote the final standard tableau in that sequence by
$\mathsf{T}(\pi)$. One easily reverses the process, obtaining a
bijection
\[
\pi \mapsto \mathsf{T}(\pi)
\]
between walks on the Bratteli diagram and the set of standard
tableaux. Thus we could just as well index our maximal vectors by
standard tableaux, but we prefer to use walks in the Bratteli diagram
for that purpose.



\section{Preliminaries}\label{s:pre}\noindent
We establish basic notational conventions and definitions.

\subsection*{Root datum}\noindent
Let $\fgl_n = \fgl_n(\C)$ be the general linear Lie algebra of $n
\times n$ matrices over~$\C$. Let $\{ e_{ij} \}_{i,j = 1,\dots, n}$ be
the standard basis of matrix units.  As a Lie algebra, $\fgl_n$ is
generated by the elements
\[
H_1, \dots, H_n \quad\text{and}\quad e_i, \dots, e_{n-1},\quad
f_i, \dots, f_{n-1}
\]
where $H_i = e_{ii}$, $e_i= e_{i,i+1}$, and $f_i= e_{i+1,i}$.  We fix
the Cartan subalgebra
\[
\mathfrak{h}= \textstyle \sum_{i=1}^n \C H_i \text{ and set }
\mathfrak{h}^* = \Hom_\C(\mathfrak{h}, \C)
\]
with basis $\{\ep_1, \dots, \ep_n\}$ dual to $\{H_1, \dots, H_n\}$.
Let $\bil{-}{-}: \mathfrak{h} \times \mathfrak{h}^* \to \C$ be the
dual pairing (evaluation) given by $\bil{H_i}{\ep_j} = \ep_j(H_i) =
\delta_{ij}$. For each $i = 1, \dots, n-1$ set
\begin{equation*} 
  \alpha^\vee_i = H_i - H_{i+1} \quad \text{and} \quad
  \alpha_i = \ep_i - \ep_{i+1} .
\end{equation*}
Let $\X^\vee = \sum_{i=1}^n \Z H_i$ (the coweight lattice) and $\X =
\sum_{i=1}^n \Z \ep_i$ (the weight lattice). The pairing $\bil{-}{-}:
\mathfrak{h} \times \mathfrak{h}^* \to \C$ restricts to a perfect
pairing
\[
\bil{-}{-}: \X^\vee \times \X \to \Z
\]
of free abelian groups.  Since the pairing is perfect, we can (and
will) identify $\X$ with the dual $\Hom_\Z(\X^\vee,\Z)$ and also
identify $\X^\vee$ with the dual $\Hom_\Z(\X,\Z)$. Thus, we will write
\begin{equation}
  \bil{\mu}{\lambda} = \mu(\lambda) = \lambda(\mu)
\end{equation}
interchangeably, for any $\mu \in \X^\vee$, $\lambda \in \X$.  This
nonstandard notation makes certain calculations more palatable. Set
\begin{align*}
\Pi^\vee &= \{ \alpha^\vee_1, \dots, \alpha^\vee_{n-1}\}
\quad\text{(simple coroots)}, \\ \Pi &= \{ \alpha_1, \dots,
\alpha_{n-1}\}\quad \text{(simple roots)}.
\end{align*}
The quadruple $(\X, \Pi, \X^\vee, \Pi^\vee)$ defines the usual root
datum associated with $\fgl_n$, in the sense of \cite{Lusztig}. The
associated Cartan matrix is $(\mathsf{a}_{ij})_{i,j=1}^{n-1}$, of type
$\mathsf{A}_{n-1}$, where $\mathsf{a}_{ij} = \alpha^\vee_i(\alpha_j)$.

\begin{rmk}
The notion of a root datum goes back to Demazure; see
e.g.\ \cites{Jantzen:RAGS,Springer}. It was generalized by Lusztig to
include the Kac--Moody Lie algebras. A similar concept is called a
``Cartan datum'' in \cite{Hong-Kang}.
\end{rmk}

\subsection*{The algebra $\UU_q = \UU_q(\fgl_n)$}
Let $\Bbbk$ be a field. We fix an element $q \ne 0$ in $\Bbbk$.
\emph{Throughout this paper, we assume that $q$ is not a root of
unity.}  The root datum $(\X, \Pi, \X^\vee, \Pi^\vee)$ defines the
quantized enveloping algebra $\UU_q(\fgl_n)$,%
\footnote{The algebra $\UU_q(\fgl_n)$ first appeared in \cite{Jimbo},
but (following Lusztig) we have slightly altered relation~\eqref{U2}.
Jimbo's version of the algebra also appears in \cite{KS}.}
as the $\Bbbk$-algebra generated by the elements $E_i$, $F_i$ ($i \in
\{1, \dots, n-1\}$) and $K_i$, $K_i^{-1}$ ($i \in \{1, \dots, n\}$)
subject to the defining relations:
\begin{gather*}
  K_i K_j = K_j K_i , \quad K_i K_i^{-1} = 1 = K_i^{-1}K_i
  \tag{U1}\label{U1}\\
  E_i F_j - F_j E_i = \delta_{ij} \frac{\tK_i - \tK_i^{-1}}{q-q^{-1}}
  \quad (\text{where } \tK_i = K_i K_{i+1}^{-1}) \tag{U2}\label{U2}\\
  K_i E_j = q^{H_i(\alpha_j)} E_j K_i , \quad
  K_i F_j = q^{-H_i(\alpha_j)} F_j K_i \tag{U3}\label{U3}\\
  E_i^2 E_j - (q+q^{-1}) E_iE_jE_i + E_jE_i^2 = 0
  \quad\text{ if } |i-j|=1 \tag{U4}\label{U4}\\
  E_iE_j = E_jE_i \quad\text{ if } |i-j| > 1 \tag{U5}\label{U5}\\
  F_i^2 F_j - (q+q^{-1}) F_iF_jF_i + F_jF_i^2 = 0
  \quad\text{ if } |i-j|=1 \tag{U6}\label{U6}\\
  F_iF_j = F_jF_i \quad\text{ if } |i-j| > 1 . \tag{U7}\label{U7}
\end{gather*}
For any $\mu = \sum_{i=1}^n m_i H_i$ in $\X^\vee$, we set $K_\mu =
\prod_{i=1}^n K_i^{m_i}$. (In particular, $K_i = K_{H_i}$ and $\tK_i = K_i
K_{i+1}^{-1} = K_{\alpha_i^\vee}$.)  Note that relations
\eqref{U3} are equivalent to the relations
\[
  K_\mu E_j = q^{\mu(\alpha_j)} E_j K_\mu , \quad
  K_\mu F_j = q^{-\mu(\alpha_j)} F_j K_\mu \tag{U$3'$}\label{U3'}
\]
for any $\mu \in \X^\vee$, $j = 1, \dots, n-1$. We have 
$\UU_q = \UU_q^- \UU_q^0 \UU_q^+$ (the triangular decomposition)
where $\UU_q^-$ (resp., $\UU_q^+$) is the subalgebra generated by 
the $F_i$ (resp., the $E_i$) and $\UU_q^0$ is the subalgebra generated by
the $K_j$. 

The algebra $\UU_q = \UU_q(\fgl_n)$ is a Hopf algebra (in more than one
way). We will use the coproduct $\Delta: \UU_q \to \UU_q \otimes \UU_q$ is
defined on generators by the rules:
\begin{equation}\label{e:Delta}
\begin{aligned}
  \Delta(E_i) &= E_i \otimes 1 + \tK_i \otimes E_i\\
  \Delta(F_i) &= F_i \otimes \tK_i^{-1} + 1 \otimes F_i \\
  \Delta(K_\mu) &= K_\mu \otimes K_\mu 
\end{aligned}
\end{equation}
for any $i = 1, \dots, n-1$ and any $\mu$ in $\X^\vee$. This makes a
tensor product $V \otimes V'$ of $\UU_q$-modules into a
$\UU_q$-module. Specifically, $V \otimes V'$ is naturally a $\UU_q
\otimes \UU_q$-module. To make it a $\UU_q$-module, compose the
corresponding representation
\[
\UU_q \otimes \UU_q \to \End(V \otimes V')
\]
with $\Delta: \UU_q \to \UU_q \otimes \UU_q$. We will not need the
counit or the antipode from the Hopf algebra structure, so we omit
their definition, which can be found for instance in
\cites{Lusztig,Jantzen:LQG}.

\subsection*{The algebra $\UU_q(\fsl_n)$.}
The subalgebra of $\UU_q(\fgl_n)$ generated by the $E_i$, $F_i$, and
the $\tK_i^{\pm 1}$ for $i = 1, \dots, n-1$ is isomorphic to
$\UU_q(\fsl_n)$.  Restricting the Hopf algebra maps to this subalgebra
makes it a Hopf algebra.

\subsection*{Specialization}\noindent
We write $\UU_x = \UU_x(\fgl_n)$ for the algebra
defined by the generators and relations \eqref{U1}--\eqref{U7}, in the
special case where $q=x$ is an indeterminate and $\Bbbk = \Q(x)$ is
the field of rational functions in $x$. The algebra $\UU_x$ is often
called the \emph{generic} algebra. Let $\A = \Z[x,x^{-1}]$ be the ring
of Laurent polynomials in $x$. Define
\[
[m]_x = \frac{x^{m} - x^{-m}}{x - x^{-1}} . 
\]
This is the (balanced form of) the Gaussian integer corresponding to
the integer $m$. We have $[0]_x = 0$, $[1]_x = 1$, and $[-m]_x =
-[m]_x$ for any $m$. If $m \ge 1$ then $[m]_x = \sum_{t=0}^{m-1}
x^{m-1-2t}$. Thus $[m]_x \in \A$ for any $m \in \Z$.  Write
\[
[m]_x^! = [1]_x [2]_x \cdots [m]_x \quad \text{ for any $m \ge 0$}.
\]
As usual, we agree that $[0]_x^! = 1$. For any nonnegative integer
$m$, define the quantum divided powers of the generators by
\[
F_i^{(m)} = \frac{F_i^m}{[m]_x^!}, \quad E_i^{(m)} =
\frac{E_i^m}{[m]_x^!}.
\]
There are two standard $\A$-forms of $\UU_x$, due respectively to
DeConcini--Kac and Lusztig, but we will only need the Lusztig version
$\UU_\A$, defined by
\begin{equation*}
  \UU_\A = \A\text{-subalgebra of $\UU_x$ generated by $E_i^{(a)}$,
    $F_i^{(b)}$, and the $K_j^{\pm 1}$}
\end{equation*}
for all $i = 1, \dots, n-1$, $j = 1, \dots, n$, and $a,b \ge 0$. This
is a quantum analogue of the Kostant $\Z$-form of the $\Q$-enveloping
algebra of the Lie algebra $\fgl_n(\Q)$. It is easily checked that
$\UU_\A \otimes_\A \Q(x)$ is generated by the elements $E_i \otimes
1$, $F_i \otimes 1$, and $K_j \otimes 1$, and that the algebra map
sending $E_i \otimes 1 \mapsto E_i$, $F_i \otimes 1 \mapsto F_i$, and
$K_j^{\pm 1} \otimes 1 \mapsto K_j^{\pm 1}$ defines an algebra
isomorphism $\UU_\A \otimes_\A \Q(x) \cong \UU_x$, so $\UU_\A$ really
is an $\A$-form of $\UU_x$.

For any commutative ring $\Bbbk$ and any invertible element $q$ in
$\Bbbk$, one regards $\Bbbk$ as an $\A$-algebra by means of the
natural ring morphism $\text{ev}_q: \A \to \Bbbk$ sending $x$ to $q$
(and $x^{-1}$ to $q^{-1}$). Define $\UU_q = \UU_q(\fgl_n)$ to be
specialized algebra
\begin{equation}\label{e:U_q}
\UU_q = \UU_\A \otimes_\A \Bbbk.
\end{equation}
At first glance this looks like a conflict of notation, because we
previously defined $\UU_q$ by generators and relations.  But we are
assuming that $q$ is not a root of unity, and in that case it is well
known that the two definitions give isomorphic algebras; see
e.g.~\cite{Jantzen:98}*{p.~118} for a summary and \cite{Jantzen:LQG}
for the arguments.  

As a point of notation, the evaluation map $\text{ev}_q$ considered
above allows one to define the elements $[m]_q$ and $[m]_q^!$ in
$\Bbbk$ by ``evaluating $x$ to $q$.'' More precisely, we define
\begin{equation}
[m]_q = \text{ev}_{q}([m]_x) \quad \text{and} \quad [m]_q^! =
\text{ev}_{q}([m]_x^!)
\end{equation}
by taking the images of $[m]_x$ and $[m]_x^!$ under the map
$\text{ev}_q$.

\subsection*{Representations and weights}\noindent
As usual, we identify the set $\X$ of weights with $\Z^n$ by means of
the isomorphism $\sum_{i=1}^n \lambda_i \ep_i \mapsto (\lambda_1,
\dots, \lambda_n)$.  If $M$ is a $\UU_q$-module of type $\mathbf{1}$,
then $M = \bigoplus_{\lambda \in \X} M_\lambda$, where the weight
space $M_\lambda$ is given by
\begin{equation*}
M_\lambda = \{ v\in M \mid K_i v = q^{H_i(\lambda)} v = q^{\lambda_i}
v, \text{ all } i = 1, \dots, n \}.
\end{equation*}
More generally, if $v \in M_\lambda$ then $K_\mu v = q^{\mu(\lambda)}
v$, for all $\mu \in \X^\vee$; in particular,
\begin{equation}
\tK_i v = K_{\alpha^\vee_i} v = q^{\alpha^\vee_i(\lambda)} v =
q^{\lambda_i - \lambda_{i+1}} v
\end{equation}
for all $i = 1, \dots, n-1$.


We have $\X^+ = \{\lambda \in \X \mid \alpha_i^\vee(\lambda) \ge 0
\text{ for all $i = 1, \dots, n-1$}\}$.  As $\alpha_i^\vee(\lambda)
= \lambda_i - \lambda_{i+1}$, this means that under the identification
of $\X$ with $\Z^n$,
\[
\X^+ = (\lambda_1, \dots, \lambda_n) \in \Z^n \mid \lambda_1 \ge
\cdots \ge \lambda_{n-1} \ge \lambda_n \}.
\]
We let $\Lambda = \N^n$ be the set of polynomial weights.  The set of
dominant polynomial weights is $\Lambda^+ = \Lambda \cap \X^+$. Thus
\[
\Lambda^+ = \{(\lambda_1, \dots, \lambda_n) \in \Z^n \mid \lambda_1
\ge \cdots \ge \lambda_{n-1} \ge \lambda_n \ge 0 \}.
\]
Dropping trailing zeros in its elements, the set $\Lambda^+$ naturally
identifies with the set of \emph{partitions} into not more than $n$
parts. If we need to vary $n$, we sometimes write $\Lambda =
\Lambda(n)$ and $\Lambda^+ = \Lambda^+(n)$.  We have
\[
\Lambda(n) = \textstyle \bigsqcup_{r\ge0} \Lambda(n,r)
\]
where $\Lambda(n,r) = \{\lambda \in \N^n \mid \sum_i
\lambda_i = r \}$ is the inverse image of $r$ under the map $\Z^n \to
\Z$ given by $\lambda \mapsto \sum_i \lambda_i$.
Intersecting this decomposition with $\Lambda^+(n)$
we get a corresponding decomposition
\[
\Lambda^+(n) = \textstyle \bigsqcup_{r\ge0} \Lambda^+(n,r)
\]
where $\Lambda^+(n,r) = \Lambda(n,r) \cap \Lambda^+(n)$. The set
$\Lambda^+(n,r)$ is the set of partitions of $r$ into not more than
$n$ parts.

We will always identify a partition $\lambda$ with its Young diagram,
consisting of the set of nodes $(i,j) \in \N \times \N$ satisfying $j
\le \lambda_i$. (There are $\lambda_j$ nodes in the $j$th row for each
$j$.)  We identify nodes with boxes in the usual fashion. Following
Kleshchev, we say that a node at position $(j,\lambda_j+1)$ is
\emph{addable} if $\lambda+\varepsilon_j$ is a partition.

\section{Schur--Weyl duality}\label{s:SWD}\noindent
Let $\Mod{\UU_q}$ be the category of finite dimensional
$\UU_q$-modules of type $\mathbf{1}$ admitting a weight space
decomposition.  From now on, all of our calculations take place in
this category.  Let $V_q(\lambda)$ be the (type $\mathbf{1}$) Weyl
module of highest weight $\lambda$. Then $V_q(\lambda)$ is an object
in $\Mod{\UU_q}$.  Since we are assuming that $q$ is not a root of
unity, $V_q(\lambda)$ is a simple module. In fact, we have the
following well known result.

\begin{thm}[Lusztig, Andersen--Polo--Wen, etc]\label{t:Uq-Mod}
Suppose that $0 \ne q \in \Bbbk$ is not a root of unity. Then:
\begin{enumerate}
\item $\Mod{\UU_q}$ is a semisimple braided tensor category. 
\item $\{ V_q(\lambda) \mid \lambda \in \X^+ \}$ is a complete set of
  simple $\UU_q$-modules in $\Mod{\UU_q}$.
\item The character of $V_q(\lambda)$ is given by Weyl's character
  formula.
\end{enumerate}
\end{thm}

\begin{proof}
A proof can be found in Jantzen's book \cite{Jantzen:LQG}. 
\end{proof}

Let $V_q = V_q(1)$ be the vector representation of $\UU_q =
\UU_q(\fgl_n)$.  Let $\{v_i\}_{i=1}^n$ be the standard basis of weight
vectors for $V_q$, where the weight of $v_i$ is $\ep_i$, for each $i =
i, \dots, n$. Then
\begin{equation}
  K_{j} v_i = q^{H_j(\ep_i)} v_i = q^{\delta_{ij}} v_i \label{e:K_ep-action}
\end{equation}  
for all $i,j \in \{1,\dots, n\}$ and thus (as $\tK_j = K_{\alpha_j^\vee}$)
\begin{equation}\label{e:tKj-action}
\tK_j v_i = q^{\alpha^\vee_j(\ep_i)} v_i = q^{\delta_{i,j} -
  \delta_{i,j+1}} v_i \quad (j = 1, \dots, n-1)
\end{equation}
For each $1 \le i \le n-1$, the operator $F_i$ (resp., $E_i$) sends
$v_{i}$ to $v_{i+1}$ (resp., $v_{i+1}$ to $v_{i}$) and sends all other
$v_k$ to $0$.


Consider the $r$th tensor power $V_q^{\otimes r}$.  The set $\{v_\aaa
\mid \aaa \in I(n,r)\}$ is a $\Bbbk$-basis of $V_q^{\otimes r}$, where
$I(n,r) = \{1,\dots, n\}^r$ and where $v_\aaa = v_{a_1} \otimes \cdots
\otimes v_{a_r}$ for each $\aaa = (a_1, \dots, a_r)$ in $I(n,r)$. 
The symmetric group $\Sym_r$ acts on $I(n,r)$ on the right by $\aaa
\cdot w = (a_{w(a_1)}, \dots, a_{w(a_r)})$. In particular,
\[
\aaa \cdot s_i = (a_1, \dots, a_{i-1}, a_{i+1}, a_i, a_{i+2}, \dots,
a_r),
\]
is the result of interchanging the entries in places $i$ and $i+1$ of
the sequence $\aaa = (a_1, \dots, a_r)$ in $I(n,r)$, where $s_i =
(i,i+1)$ is the transposition in $\Sym_r$ interchanging $i$ with
$i+1$. Then $v_\aaa w = v_{\aaa \cdot w}$ (for $w \in \Sym_r$) defines
the usual place-permutation action of $\Sym_r$ on $V_q^{\otimes r}$.

Let $\HH_q = \HH_q(\Sym_r)$ be the (balanced form of) Iwahori--Hecke
algebra of the symmetric group $\Sym_r$. This is the $\Bbbk$-algebra
defined by generators $T_1, \dots, T_{r-1}$ satisfying the quadratic
relation
\begin{equation}\label{e:quad}
  (T_i - q)(T_i + q^{-1}) = 0 \qquad \text{(for all $i$)}
\end{equation}
along with the usual type $\sf{A}$ braid relations; that is, the quotient of
Artin's braid group algebra by relation \eqref{e:quad}.\footnote{Our
version of $\HH_q(\Sym_r)$ follows the normalization convention of
\cite{Lu:unequal}; cf.~\cite{Lu:83}. It differs from the original
version (e.g., \cites{KL,Jimbo,DJ:Hecke,DJ:Hecke-87,Mathas}) although
the two versions are isomorphic under suitable assumptions on the
ground ring.}
The algebra $\HH_q$
acts on $V_q^{\otimes r}$ on the right by
\begin{equation}
v_\aaa T_i = 
\begin{cases}
  q v_\aaa & \text{ if } a_i = a_{i+1} \\
  v_{\aaa \cdot s_i} & \text{ if } a_i > a_{i+1} \\
  v_{\aaa \cdot s_i} + (q-q^{-1}) v_\aaa & \text{ if } a_i < a_{i+1} .
\end{cases}
\end{equation}
(If $q=1$ then the above action coincides with the usual
place-permutation action of $\Sym_r$.)  The following
was first observed in \cite{Jimbo}.

\begin{lem}[Jimbo]\label{l:commuting-actions}
The actions of $\HH_q$ and $\UU_q$ on $V_q^{\otimes r}$ commute.
\end{lem}

\begin{proof}
For any $j = 1, \dots, r-1$, the action of $T_j$ on $V_q^{\otimes r}$
may be rewritten in the form $T_j = I^{\otimes (j-1)} \otimes T
\otimes I^{\otimes (r-1-j)}$, where $I$ is the identity operator on
$V_q$ and where $T: V_q \otimes V_q \to V_q \otimes V_q$ is given by
\[
T(v_s \otimes v_t) =
\begin{cases}
  q \, v_s \otimes v_t & \text{ if } s=t \\
  (q-q^{-1})\, v_s \otimes v_t + v_t \otimes v_s & \text{ if } s<t \\
  v_t \otimes v_s & \text{ if } s>t.
\end{cases}
\]
For any $i = 1, \dots, n-1$, the action of $E_i$ on $V_q^{\otimes r}$ is
given by
\[
\Delta^{(r-1)}(E_i) = \sum_{s = 1}^{r-1} \tK_i^{\otimes (s-1)} \otimes
(E_i \otimes I + \tK_i \otimes E_i) \otimes I^{\otimes (r-1-s)} .
\]
This is obtained by iterating the coproduct $\Delta$ defined in
equation~\eqref{e:Delta}.  One easily checks that $(\tK_i \otimes
\tK_i)T = T(\tK_i \otimes \tK_i)$.  Obviously $(I \otimes I)T = T(I
\otimes I)$. From this it is clear that $E_j$ and $T_i$ on tensor
space $V_q^{\otimes r}$ commute if and only if
\[
T (\Delta E_i) = (\Delta E_i)T
\]
as operators on $V_q \otimes V_q$.  By~\eqref{e:Delta}, the above equality
holds if and only if
\[
T(E_i \otimes I + \tK_i \otimes E_i) = (E_i \otimes I + \tK_i \otimes
E_i)T
\]
as operators on $V_q \otimes V_q$. That this latter identity holds is
easy to check directly, using the fact that $E_i(v_{i+1}) = v_i$ and
$E_i(v_t) = 0$ for all $t \ne i+1$. The entirely similar argument for
$F_i$ in place of $E_i$ is left to the reader.
\end{proof}

\begin{rmk}
(i) Jimbo used a slightly different version of $\HH_q$, in which the
  $T_i$ satisfy \eqref{e:quad} but with $q$ and $q^{-1}$ interchanged.

(ii) The operator $T$ used above is closely related to the $R$-matrix
  formalism in connection with the Yang--Baxter equation in statistical
  mechanics.
\end{rmk}

\begin{thm}[Jimbo]\label{t:SWD}
Suppose that $0 \ne q \in \Bbbk$ is not a root of unity.
The commuting actions of $\UU_q$ and $\HH_q$ induce morphisms
\[
\UU_q \to \End^{}_{\HH_q}(V_q^{\otimes r}) \quad\text{and} \quad \HH_q
\to \End^{}_{\UU_q}(V_q^{\otimes r}),
\]
each of which is surjective. In other words, the image of
each action is equal to the full centralizer of the other.
\end{thm}

\begin{proof}
Since $q$ is not a root of unity, the category $\Mod{\HH_q}$ of finite
dimensional $\HH_q$-modules is semisimple. Dipper and James
\cites{DJ:Hecke,DJ:Hecke-87} constructed $q$-analogues
$\Specht_q^\lambda$ of Specht modules and showed that
$\{\Specht_q^\lambda \mid \lambda \vdash r\}$ is a complete set of
simple modules in $\Mod{\HH_q}$.  Furthermore, they showed that
$\dim_\Bbbk \Specht_q^\lambda = \dim_\C \Specht^\lambda$, where
$\Specht^\lambda$ is the classical Specht module for the symmetric
group $\C[\Sym_r]$. (This also follows from character theory
\cites{Ram,Geck-Pfeiffer} for $\HH_q$; where it is known that
evaluating $\ch \Specht_q^\lambda$ at $q=1$ gives $\ch
\Specht^\lambda$.)

Thus we have a bijection $\Specht_q^\lambda \mapsto \Specht^\lambda$
mapping simple objects in $\Mod{\HH_q}$ onto simple objects in
$\Mod{\C[\Sym_r]}$. This bijection preserves multiplicities in
semisimple decompositions. By Theorem~\ref{t:Uq-Mod}, we have a
similar bijection $V_q(\lambda) \mapsto V(\lambda)$, where
$V(\lambda)$ is the classical Weyl module for the Lie algebra
$\fgl_n(\C)$. This bijection also preserves multiplicities in semisimple
decompositions. As a result, we have the semisimple decompositions
\begin{equation}\label{e:ss-decomps}
V_q^{\otimes r} \cong \textstyle \bigoplus_\lambda
\mult_\lambda^{\sst} \Specht_q^\lambda \quad \text{and} \quad
V_q^{\otimes r} \cong \textstyle \bigoplus_\lambda \mult_\lambda^{\st}
V_q(\lambda)
\end{equation}
in $\Mod{\HH_q}$ and $\Mod{\UU_q}$, respectively, where
$\mult_\lambda^{\sst}$ and $\mult_\lambda^{\st}$ are respectively
equal to the number of semistandard and standard tableaux of shape
$\lambda$. In both decompositions, the index $\lambda$ ranges over the
set of partitions of $r$ into not more than $n$ parts. We know these
multiplicities because of classical Schur--Weyl duality.

To finish the proof, it is enough to prove one of the claimed
surjectivities, because we then get the other by the standard
double-centralizer property for semisimple algebras. We will now argue
that the map $\HH_q \to \End^{}_{\UU_q}(V_q^{\otimes r})$ is
surjective. Since $\HH_q$ is (split) semisimple,
\[
\HH_q \cong \textstyle \bigoplus_{\lambda \vdash
  r} \End^{}_\Bbbk(\Specht_q^\lambda).
\]
As $\mult_\lambda^{\sst} > 0$ for all $\lambda \in \Lambda^+(n,r)$, it
follows from the first decomposition in \eqref{e:ss-decomps} that the
kernel of the morphism $\HH_q \to \End^{}_{\UU_q}(V_q^{\otimes r})$ is
isomorphic to the direct sum of all $\End^{}_\Bbbk(\Specht_q^\lambda)$
such that $\lambda$ has strictly more than $n$ parts, so its image is
isomorphic to the direct sum of all $\End^{}_\Bbbk(\Specht_q^\lambda)$
over $\lambda \in \Lambda^+(n,r)$. But $\dim_\Bbbk \Specht_q^\lambda =
\mult_\lambda^{\st}$, so the dimension of the image of the morphism is
equal to $\sum_{\lambda \in \Lambda^+(n,r)} (\mult_\lambda^{\st})^2$.

On the other hand, it follows from Schur's lemma and the second
decomposition in \eqref{e:ss-decomps} that the dimension of the
centralizer algebra $\End^{}_{\UU_q}(V_q^{\otimes r})$ is given by the
same sum of squares, so we are finished.
\end{proof}

\begin{cor}\label{c:SWD}
Suppose that $0 \ne q \in \Bbbk$ is not a root of unity.  Then
\begin{enumerate}
\item $V_q^{\otimes r} \cong \bigoplus_{\lambda \in \Lambda^+(n,r)}
  V_q(\lambda) \otimes \Specht_q^\lambda$ as
  $(\UU_q,\HH_q)$-bimodules.
\end{enumerate}
For any $\lambda \in \Lambda^+(n,r)$ there are isomorphisms
\begin{enumerate}\setcounter{enumi}{1}%
\item $V_q(\lambda) \cong \Hom^{}_{\HH_q}(\Specht_q^\lambda, V_q^{\otimes
  r})$, as $\UU_q$-modules.
\item $\Specht_q^\lambda \cong \Hom^{}_{\UU_q}(V_q^\lambda, V_q^{\otimes
  r})$, as $\HH_q$-modules.
\end{enumerate}
\end{cor}

\begin{proof}
Part (a) is a standard consequence of the double-centralizer property.
Parts (b) and (c) follow immediately from (a).
\end{proof}

\begin{rmk}
(i) Theorem \ref{t:SWD} was announced in \cite{Jimbo}, for the generic
  case (where $\Bbbk = \C(x)$ and $q = x$).  A more general version,
  which includes the case where $q$ is a root of unity, was proved in
  \cite{DPS}; see also \cite{Martin} and
  \cite{Donkin:qSchur}*{\S4.7}. Our method is very different from the
  methods used in those references. Note that the authors of
  \cite{DPS} replace $q$ by a square root $q^{1/2}$; this is because
  they work with a slightly different normalization of $\HH_q$.
  
(ii) $\End^{}_{\HH_q}(V_q^{\otimes r})$ is isomorphic to the $q$-Schur
algebra $\mathbf{S}_q(n,r)$ introduced by Dipper and James
\cite{DJ1}. This is obvious in the semisimple case, but it works in
general \cites{DJ2,Grojnowski-Lusztig}. A geometric construction of
the $q$-Schur algebra was given in \cite{BLM}.

(iii) Dipper and James \cite{DJ:Hecke} showed that $\HH_q(\Sym_r)$ is
semisimple if and only if $[r]_q^! \ne 0$ in $\Bbbk$. Tensor space
$V_q^{\otimes r}$ is still semisimple under this hypothesis, both in
$\Mod{\HH_q}$ and $\Mod{\UU_q}$. Hence the above results hold in this
slightly more general setting. As a consequence, we see that the
$q$-Schur algebra $\mathbf{S}_q(n,r)$ is split semisimple if $[r]_q^! \ne
0$ in $\Bbbk$. (The converse of this implication is false
\cite{EN}*{Thm.~1.3(A)}.)

%
%

(iv) If we replace $\UU_q(\fgl_n)$ by $\UU_q(\fsl_n)$ then all the
results of this section still hold.
\end{rmk}

\section{Coxeter monomials}\label{s:Coxeter}\noindent
Let $W_n = W(\mathrm{A}_{n-1})$ be the Weyl group associated to the
root datum, and denote its generating set of simple reflections by
$s_1, \dots, s_{n-1}$. As usual, we identify $s_i$ with the
transposition that swaps $i$ with $i+1$. The group $W_n$ is isomorphic
to the symmetric group on $n$ letters.  Recall (see e.g.,
\cite{Humphreys}*{\S3.16}) that a \emph{Coxeter element} of $W_n$ is
an element which can be written as a product of generators in which
each generator appears exactly once.  So there are exactly $2^{n-2}$
distinct Coxeter elements in $W_n$.

\begin{rmk}
Coxeter elements in $W_n$ are important examples of fully commutative
elements.  Recall that an element of $W_n$ is fully commutative if any
reduced expression for $w$ is obtainable from any other by applying
(in adjacent positions) commutation relations of the form $s_is_j =
s_js_i$ where $|i-j|>1$. (This was generalized to Coxeter groups in
\cite{Stembridge}.) In $W_n$, fully commutative elements are the same
as $321$-avoiding permutations \cite{BJS}; their number is 
the $n$th Catalan number $\frac{1}{n+1}\binom{2n}{n}$.
\end{rmk}

Given a reduced expression $w = s_{i_1} \cdots s_{i_{n-1}}$ for a
Coxeter element $w$ in $W_n$, we define
\[
F_w = F_{i_1} \cdots F_{i_{n-1}}.
\]
This element belongs to $\UU^-_q(\fgl_n)$, and its definition is
independent of the choice of reduced expression for $w$. We call such
elements \emph{Coxeter monomials}. Evidently, we have a natural
bijection between the Coxeter elements in $W_n$ and the Coxeter
monomials in $\UU^-_q(\fgl_n)$.  

\begin{example}\label{ex:Coxeter}
The eight Coxeter monomials in $\UU^-_q(\fgl_5)$ are enumerated below:
\[
F_{1234}, F_{2341}, F_{1342}, F_{3421}, F_{1243}, F_{2431}, F_{1432}, F_{4321}
\]
where $F_{1234}$ is shorthand for $F_1F_2F_3F_4$, and so forth.  There
are $4! = 24$ distinct orderings of the numbers $1,2,3,4$ but just $8$
Coxeter monomials. Each of the $24$ possible orderings is equal to one
of the elements listed above, by applying a sequence of commutation
relations of the form \eqref{U7}.
\end{example}

We need to consider \emph{shifted} Coxeter monomials.  If $F_w$ is a
Coxeter monomial then its shift $F_w^+$ is defined by replacing each
$F_i$ by $F_{i+1}$. The shift of a Coxeter monomial in
$\UU^-_q(\fgl_n)$ belongs to $\UU^-_q(\fgl_{n+1})$.  We can iterate
shifting more than once to obtain additional shifted Coxeter
monomials. We extend the same notation to Coxeter elements: if $w$ is
a Coxeter element then $w^+$ denotes the result of replacing each
$s_j$ by $s_{j+1}$ in a reduced expression for $w$. We have $F_w^+ =
F_{w^+}$.  Shifting provides a simple recursive method to generate
(reduced expressions for) all Coxeter elements, and thus generate all
Coxeter monomials, as follows. We define $I_1 = \{s_1\}$ and, for any
$n \ge 1$, we define $I_{n}$ by the disjoint union
\begin{equation}\label{e:I_n}
  I_{n} = \{ s_1 w^+ \mid w \in I_{n-1} \} \sqcup \{ w^+ s_1 \mid
  w \in I_{n-1} \}.
\end{equation}
Then we have the following result, the proof of which is clear from
the preceding analysis.

\begin{lem}\label{l:Coxeter}
For any $n\ge 2$, the set $I_{n-1}$ is the set of distinct Coxeter
elements in $W_n$, and the set $\{ F_w \mid w \in I_{n-1} \}$ is the
set of distinct Coxeter monomials in $\UU^-_q(\fgl_{n})$.
\end{lem}

\begin{rmk}\label{r:dist-reps}
This result constructs a set of \emph{distinguished} reduced
expressions for the Coxeter elements, as illustrated in
Example~\ref{ex:Coxeter} above, and similarly for the Coxeter
monomials. Such reduced expressions begin or end with $s_1$ (resp.,
$F_1$).
\end{rmk}

\section{The $\Psi$ operators}\label{s:Psi}\noindent
For the rest of the paper, we fix $0 \ne q \in \Bbbk$ and simplify
notation by suppressing the subscript $q$ in $[a]_q$, writing $[a] =
[a]_q$.  In this section, we inductively define a sequence $\Psi_1,
\Psi_2, \dots$, depending on a given dominant weight $\lambda$, such
that $\Psi_1, \dots, \Psi_n$ belong to $\UU_q^-(\fgl_{n+1})$ for all
$n$, and each $\Psi_j$ is a linear combination of Coxeter monomials in
the variables $F_1, \dots, F_{j-1}$. We regard $\UU_q^-(\fgl_n)$ as
embedded in $\UU_q^-(\fgl_{n+1})$ via the map taking $F_j \mapsto F_j$
if $j<n$. For a given $\lambda$, $\Psi_j$ first appears in
$\UU_q^-(\fgl_{j+1})$ if it is defined, and once defined, it maintains
the same value in $\UU_q^-(\fgl_n)$ for all $n \ge j+1$. We emphasize
that the $\Psi_j$ depend on $\lambda$ although we usually suppress
that dependence in the notation.
In Section~\ref{s:q=1} we explain how the $\Psi_j$ are related to
certain $q$-analogues of Carter's lowering operators.

\begin{defn}\label{d:Psi}
Fix a partition $\lambda$ in $\X^+$ and regard it as an infinite
sequence by appending zeros.

(i) We define integers $c_{n} = c_{n,\lambda}$ and $d_{n} =
d_{n,\lambda}$, depending on the given $\lambda$, by setting $c_1=0$ and
\[
\begin{alignedat}{3}
  c_n &=
  (\alpha_2+\cdots+\alpha_n)^\vee(\lambda) + n-1 &&\quad (\text{for all
  } n \ge 2), \\ d_n &=
  (\alpha_1+\cdots+\alpha_n)^\vee(\lambda) + n-1 && \quad
  (\text{for all } n \ge 1).
\end{alignedat}
\]
Notice that $d_n = c_n+\alpha^\vee_1(\lambda)$ for all $n \ge 1$ and
$d^+_{n-1} + 1 = c_n$ for all $n \ge 2$, where here the shift operator
$+$ replaces each $\alpha_j$ by $\alpha_{j+1}$.
We have
\[
\begin{alignedat}{3}
  c_n &= \lambda_2 - \lambda_{n+1} + n-1 &&\quad (\text{for all
  } n \ge 2), \\ d_n &=
  \lambda_1 - \lambda_{n+1} + n-1 && \quad
  (\text{for all } n \ge 1).
\end{alignedat}
\]
This makes it clear that $c_n$ is zero if and only if $n=1$, while
$d_n=0$ if and only if $n=1$ and $\lambda_1-\lambda_2=0$. 

(ii) Set $\Psi_0 = 1$. Define a sequence of operators $\Psi_1 =
\Psi_{1,\lambda}, \Psi_2 = \Psi_{2,\lambda}, \dots$, depending on
$\lambda$, such that $\Psi_n$ lies in the negative part
$\UU^-_q(\fgl_{n+1})$, for each $n$, by means of the recursion:
\[
\Psi_{1} = \frac{F_1}{[d_1]}, \qquad \Psi_{n} =
\frac{1}{[d_{n}]}\big([c_{n}] F_1 \Psi^+_{n-1} - [c_{n}-1]
\Psi^+_{n-1} F_1\big) \quad\text{if } n \ge 2.
\]
The $+$ superscript applied to a $\Psi_n$ means to replace each $F_j$
by $F_{j+1}$ and also replace each $\alpha_j$ by $\alpha_{j+1}$.  More
generally, a superscript of $+(j)$ means to shift $j$ times.
\end{defn}

\begin{example}\label{ex:1}
For each $j$, we set $a_j = \alpha_j^\vee(\lambda) = \lambda_j -
\lambda_{j+1}$. In the notation of Example~\ref{ex:Coxeter}, we have:
\begin{align*}
\Psi_1 & = \frac{F_1}{[d_1]} \quad \text{if $d_1 \ne 0$}.\\
\Psi_2  &= \frac{1}{[d_2][d_1^+]} \big([1+d_1^+]F_{12} - [d_1^+]F_{21}\big)
\quad \text{if $d_1^+ \ne 0$}.\\
\Psi_3  &= \frac{1}{[d_3][d_2^+][d_1^{++}]} \big(x_1F_{123} - x_2F_{132}
- x_3F_{231} + x_4F_{321} \big) \quad \text{if $d_1^{++} \ne 0$}.
\end{align*}
As $F_{13}=F_{31}$, we have $F_{312} = F_{132}$ and $F_{213} =
F_{231}$; this shows the importance of Coxeter monomials.
In the above, $d_1 = a_1$, $d_2 = a_1+a_2+1$, and $d_3 =
a_1+a_2+a_3+2$. Furthermore,
\begin{alignat*}{2}
x_1 &= [1+d_2^+][1+d_1^{++}], \qquad & x_2 &= [1+d_2^+][d_1^{++}], \\
x_3 &= [d_2^+][1+d_1^{++}], \qquad & x_4 &= [d_2^+][d_1^{++}].
\end{alignat*}
Notice that $\Psi_1$ is defined if and only if $d_1 = a_1 = \lambda_1
- \lambda_2 \ne 0$. As $d_2 \ge 1$ and $d_3 \ge 2$, we see that
$\Psi_2$ is defined if and only if $d^+_1 = a_2 = \lambda_2 -
\lambda_3 \ne 0$, and similarly, $\Psi_3$ is defined if and only if
$d_1^{++} = a_3 = \lambda_3 - \lambda_4 \ne 0$.
\end{example}

\begin{lem}\label{l:Psi-is-defined}
$\Psi_m$ is undefined if and only if $\alpha_m^\vee(\lambda) =
  \lambda_m - \lambda_{m+1} = 0$. Whenever it is defined,
  $\Psi_m \ne 0$.
\end{lem}

\begin{proof}
By induction on $m$. By Example \ref{ex:1}, $\Psi_1$ is defined if
and only if $\alpha_1^\vee(\lambda) \ne 0$. Let $m \ge 2$. By
induction, we may assume that $\Psi_{m-1}$ makes sense if and only if
$\alpha_{m-1}^\vee(\lambda) \ne 0$. Then by Definition~\ref{d:Psi},
\[
\Psi_{m} = \frac{1}{[d_{m}]}\big([c_{m}] F_1
\Psi^+_{m-1} - [c_{m}-1] \Psi^+_{m-1} F_1\big)
\]
so $\Psi_{m-1}^+$ is defined if and only if $\alpha_m^\vee(\lambda)
\ne 0$. Since neither $c_m$ nor $d_m$ can be zero (for $m \ge 1$) and
$\Psi_{m-1}^+ F_1$, $F_1 \Psi_{m-1}^+$ are linearly independent, it
follows that $\Psi_m$ is well defined and non-zero. Note that the
linear independence of $\Psi_{m-1}^+ F_1$, $F_1 \Psi_{m-1}^+$ follows
from the fact that $\UU^-(\fgl_m)$ is isomorphic to the free algebra
over $\Bbbk$ generated by $F_1, \dots, F_{m-1}$ subject only to
relations \eqref{U6} and \eqref{U7}, but relation \eqref{U6} is never
applicable.
\end{proof}

For the rest of this section, we work in the algebra
$\UU_q(\fgl_{n+1})$, and we fix a partition $\lambda = (\lambda_1,
\lambda_2, \dots)$ in $\X^+$.  We will consider a weight vector $b$
in some $\UU_q(\fgl_{n+1})$-module, which will be unspecified.

The following is the main result of this section. 

\begin{prop}\label{p:1}
If $b$ is a maximal vector of weight $\lambda$ in some module then 
\[
E_j \Psi_n b =
\begin{cases}
  \Psi^+_{n-1} b & \text{ if } j = 1 \\
  0 & \text{ if } 1 < j \le n.
\end{cases}
\]
\end{prop}

The proof of Proposition \ref{p:1} will be given at the end of the
section, after a series of lemmas.

\emph{We remind the reader that the notation $c_n$, $d_n$, and
$\Psi_n$ all depend on the chosen fixed $\lambda$.}  If $b'$ is a
weight vector of weight $\lambda' = (\lambda'_1, \lambda'_2, \dots)$
in some module then we write $\alpha^\vee_j(b') =
\alpha^\vee_j(\lambda') = \lambda'_j - \lambda'_{j+1}$, for all $j$.

\begin{lem}[$q$-integer identities]\label{l:q-int-identity}
  Let $0 \ne q \in \Bbbk$. For any $y,z \in \Z$, we have:
  \begin{enumerate}
  \item   $[y+1][z+1] - [y][z] = [y+z+1]$.
  \item   $[z] + q^{-z-1} = q^{-1} [z+1]$.
  \end{enumerate}
\end{lem}

\begin{proof}
One checks that the stated identities (with $x$ in place of $q$) hold
formally in the ring $\A=\Z[x,x^{-1}]$ of Laurent polynomials.  Then
specialize $x=q$ to get the identities in $\Bbbk$.
\end{proof}

\begin{lem}[contraction]\label{l:contraction}
  Let $b$ be a weight vector of weight $\lambda$. If $E_j b = 0$ then
  $E_jF_j b = [\alpha_j^\vee(\lambda)] b$.
\end{lem}

\begin{proof}
This is a consequence of relation~\eqref{U2}. By taking $i=j$ in that
relation we have
\[
E_j F_j b = F_j E_j b + \frac{\tK_j - \tK_j^{-1}}{q-q^{-1}} b.
\]
The result follows from the hypothesis and equation
\eqref{e:tKj-action}.
\end{proof}

\begin{lem}\label{l:E1-effect}
  Let $b'$ be a weight vector of weight $\lambda'$ in some module.
  If $E_1 b' = 0$ then
  \[
  E_1 \Psi_n b' = \frac{[c_n+\alpha_1^\vee(\lambda')]}{[d_n]} \Psi_{n-1}^+ b',
  \quad \text{ for all } n\ge 1.
  \]
\end{lem}

\begin{proof}
Applying the recursive definition of $\Psi_n$ and the contraction
lemma, we get
\begin{align*}
  E_1 \Psi_n b' &= \frac{1}{[d_n]} \big( [c_n] (E_1F_1) \Psi^+_{n-1}
  b' - [c_n-1] \Psi^+_{n-1} (E_1F_1) b' \big) \\  &=
  \frac{1}{[d_n]} \big( [c_n] [\alpha_1^\vee(\lambda')+1] - [c_n-1]
       [\alpha_1^\vee(\lambda')] \big) \Psi^+_{n-1} b' .
\end{align*}
In the above calculation, we used the fact that the weight of
$\Psi_{n-1}^+ b'$ is $\mu = \lambda' - (\alpha_2 + \cdots +
\alpha_{n})$ and $(\alpha_1^\vee, \mu) = \alpha_1^\vee(\lambda') +
1$. The result now follows from the first $q$-identity in Lemma
\ref{l:q-int-identity}.
\end{proof}

\begin{rmk}\label{r:1}
In particular, Lemma \ref{l:E1-effect} says that if $E_1 b' = 0$ and
$\alpha_1^\vee(\lambda') = \alpha_1^\vee(\lambda)$ then $E_1 \Psi_n b'
= \Psi^+_{n-1} b'$, for all $n \ge 1$.
\end{rmk}

\begin{lem}\label{l:2}
  Let $b'$ be a weight vector of weight $\lambda'$ in some module.
  If $E_2 b' = 0$
  and $\alpha_2^\vee(\lambda') = \alpha_2^\vee(\lambda)$ then
  \[
  E_2 \Psi_n b' = 0, \quad \text{ for all } n\ge 2.
  \]
\end{lem}

\begin{proof}
For the proof, we set $a'_j = \alpha_j^\vee(\lambda')$ and $a_j =
\alpha_j^\vee(\lambda)$, for all $j$.  If we replace $n$ by $n-1$ in
Lemma \ref{l:E1-effect} we obtain the statement: if $E_1 b' = 0$ then
\[
E_1 \Psi_{n-1} b' = \frac{[c_{n-1} + a'_1]}{[d_{n-1}]} \Psi_{n-2}^+ b',
\quad \text{ for all } n\ge 2.
\]
Now shift the above. We get the statement:  if $E_2 b' = 0$ then
\[
E_2 \Psi^+_{n-1} b' = \frac{[c^+_{n-1} + a'_2]}{[d^+_{n-1}]}
\Psi_{n-2}^{++} b', \quad \text{ for all } n\ge 2.
\]
But $a'_2 = a_2$ by hypothesis, and $c^+_{n-1} + a_2 = c_n -1$, so the
above becomes: if $E_2 b' = 0$ then
\[
E_2 \Psi^+_{n-1} b' = \frac{[c_n - 1]}{[d^+_{n-1}]}
\Psi_{n-2}^{++} b', \quad \text{ for all } n\ge 2.
\]
Furthermore, if $E_2 b' = 0$ then also $E_2 F_1b' = 0$. Since the
weight of $F_1b'$ is $\lambda' - \alpha_1$, we have
$\alpha_2^\vee(F_1b') = \alpha_2^\vee(\lambda' - \alpha_1) = a'_2 + 1
= a_2 +1$. Now repeat the preceding argument with $b'$ replaced by
$F_1b'$. Then we get: if $E_2 b' = 0$ then
\[
E_2 \Psi^+_{n-1} F_1 b' = \frac{[c_n]}{[d^+_{n-1}]}
\Psi_{n-2}^{++} F_1 b', \quad \text{ for all } n\ge 2.
\]
Putting the results of the last two displayed equalities into the
recursive definition of $\Psi_n b'$ then gives
\begin{align*}
E_2 \Psi_{n} b' &= \frac{1}{[d_{n}]}\big([c_{n}] F_1
E_2 \Psi^+_{n-1} b' - [c_{n}-1] E_2 \Psi^+_{n-1} F_1 b' \big) \\
&= \frac{1}{[d_{n}][d^+_{n-1}]} \big([c_{n}] [c_n - 1]
F_1 \Psi^{++}_{n-2} b' - [c_{n}-1] [c_n] \Psi^{++}_{n-2} F_1 b' \big).
\end{align*}
Finally, since $\Psi^{++}_{n-2}$ is a linear combination of products
of $F_3, \dots, F_n$ taken in various orders, it is clear that $F_1$
commutes with $\Psi^{++}_{n-2}$, so the right hand side above
evaluates to zero, as required.
\end{proof}

\begin{lem}\label{l:3}
  Let $b'$ be a weight vector of weight $\lambda'$, and let $j \ge
  2$. If $E_j b' = 0$ and $\alpha_j^\vee(\lambda') =
  \alpha_j^\vee(\lambda)$ then
  \[
  E_j \Psi_n b' = 0, \quad \text{ for all } n\ge j.
  \]
\end{lem}

\begin{proof}
In the argument, we set $a'_j = \alpha_j^\vee(\lambda')$ and $a_j =
\alpha_j^\vee(\lambda)$, for all $j$.  The proof is by induction on
$j$. The base case $j = 2$ of the induction is Lemma \ref{l:2}. For
the inductive step we assume that the result holds for some fixed $j
\ge 2$. Replace $n$ by $n-1$ in the inductive hypothesis to get:
\[
E_j b' = 0 \text{ and } a'_j = a_j \implies E_j \Psi_{n-1} b' = 0,
\text{ for all } n-1 \ge j.
\]
Shifting the above and noting that $n-1 \ge j$ is equivalent to $n \ge
j+1$ produces the implication
\[
E_{j+1} b' = 0 \text{ and } a'_{j+1} = a_{j+1} \implies E_{j+1}
\Psi^+_{n-1} b' = 0, \text{ for all } n \ge j+1.
\]
Since $j$ is at least $2$, $j+1$ is at least $3$, so $E_{j+1} (F_1b')
= F_1 E_{j+1} b' = 0$ and furthermore, $a_{j+1}(F_1b') = a'_{j+1}$,
so we may replace $b'$ by $F_1 b'$ in the preceding displayed
implication, to get:
\[
E_{j+1} b' = 0 \text{ and } a'_{j+1} = a_{j+1} \implies E_{j+1}
\Psi^+_{n-1} F_1 b' = 0, \text{ for all } n \ge j+1.
\]
The result now follows by substituting the results of the preceding two
implications into the recursive definition
\[
E_{j+1} \Psi_{n} b' = \frac{1}{[d_{n}]}\big([c_{n}] F_1
E_{j+1} \Psi^+_{n-1} b' - [c_{n}-1] E_{j+1} \Psi^+_{n-1} F_1 b' \big).
\]
As both terms on the right hand side are zero, $E_{j+1} \Psi_{n} b' =
0$, and this holds for all $n \ge j+1$. This completes the induction.
\end{proof}

We are now ready to prove the formula for $E_j \Psi_n b$ in
Proposition~\ref{p:1}, where $b$ is a maximal vector of weight
$\lambda$.  We take $b' = b$.  The first case in the desired formula
then follows from Remark \ref{r:1}, and the other case follows from
Lemma~\ref{l:3}.

\section{The $\Phi$ operators}\label{s:Phi}\noindent
Let $V_q=V_q(1)$ be the vector representation of $\UU_q(\fgl_n)$,
defined at the end of Section~\ref{s:pre}. It is immediate from
\eqref{e:Delta} that a tensor product of maximal vectors is again a
maximal vector. Thus, if $b$ is a maximal vector of weight
$\lambda$ in some $\UU_q(\fgl_n)$-module $M$, then
\begin{equation*}
  \Phi_1(b) = v_1 \otimes b
\end{equation*}
is a maximal vector of weight $\lambda+\ep_1$ in $V_q \otimes M$. We
wish to find similar elements $\Phi_2(b), \dots, \Phi_n(b)$ that will
turn out to be maximal vectors of respective weights $\lambda+\ep_2,
\dots, \lambda+\ep_n$ in $V_q \otimes M$, under suitable conditions.
To that end, we observe the following.

\begin{lem}\label{l:wt-lem}
Let $b$ be a maximal vector of weight $\lambda$.  The weight of
$\Psi_j^{+(m-j-1)} b$ is $\lambda - \ep_{m-j} + \ep_{m}$ and the
weight of $v_{m-j} \otimes\Psi_j^{+(m-j-1)} b$ is $\lambda + \ep_{m}$.
\end{lem}

\begin{proof}
The weight of $\Psi_j b$ is $\lambda - (\alpha_1 + \cdots + \alpha_j)
= \lambda - \ep_1 + \ep_{j+1}$. By shifting $m-1-j$ times, we obtain
the first statement. The second statement follows immediately
from the first.
\end{proof}

Now we define $\Phi_m(b)$ as a linear combination of the weight
vectors in Lemma~\ref{l:wt-lem}.

\begin{defn}\label{d:Phi}
Let $b$ be a maximal vector of weight $\lambda$ in some
$\UU_q(\fgl_n)$-module $M$.  For each $m=2, \dots, n$, we define a
weight vector $\Phi_m(b)$ in $V_q \otimes M$, of weight $\lambda+\ep_m$,
by:
\begin{equation*}
\Phi_m(b) = \sum_{j=0}^{m-1} (-q^{-1})^{j} v_{m-j} \otimes
\Psi_{j}^{+(m-j-1)} b ,
\end{equation*}
where the superscript $+(k)$ means to apply the $+$ operator $k$
times.  
\end{defn}

\begin{lem}
For $m\ge 2$, $\Phi_m(b)$ is defined if and only if
$\alpha_{m-1}^\vee(\lambda) \ne 0$.
\end{lem}

\begin{proof}
This follows from Lemma \ref{l:Psi-is-defined}. Each $\Psi_j$ is
defined if and only if $\alpha_j^\vee(\lambda) \ne 0$. Hence, each
shifted operator $\Psi_j^{+(m-j-1)}$ is defined if and only if
$\alpha_{m-1}^\vee(\lambda) \ne 0$.
\end{proof}

Shifting also applies to the $\Phi_m(b)$, as follows.  The notation
$\Phi^+_m(b)$ is defined by replacing each $v_k$ by $v_{k+1}$ and each
$\Psi_j$ by $\Psi_j^+$ in Definition~\ref{d:Phi}.  This notation
enables the following recursive description
\begin{equation}\label{e:Phi-recursion}
  \Phi_m(b) = \Phi^+_{m-1}(b) + (-q^{-1})^{m-1}\, v_1 \otimes
  \Psi_{m-1} b, \quad \text{for any $m \ge 1$}.
\end{equation}
In this formula it is important to treat $b$ as a formal variable.

The following is the main result of this section.

\begin{thm}\label{t:main}
If $b$ is a maximal vector of weight $\lambda$ in some
$\UU_q(\fgl_n)$-module $M$, then:
\begin{enumerate}
\item[(a)] $\Phi_1(b) = v_1 \otimes b$ is a maximal vector of weight
  $\lambda+\ep_1$ in $V_q \otimes M$.
\item[(b)] For each $2 \le m \le n$, if $\alpha_{m-1}^\vee(\lambda)
  \ne 0$ then $\Phi_m(b)$ is a maximal vector of weight
  $\lambda+\varepsilon_m$ in $V_q \otimes M$.
\end{enumerate}
\end{thm}

The proof of Theorem~\ref{t:main} is by induction on $n$, and occupies
the rest of this section. The result is trivial if $n=1$.  For the
inductive step, we assume that it holds for some $n \ge 1$.

\begin{rmk}
In the language of addable nodes, the maximal vector $\Phi_j(b)$
exists only if the node at position $(j,\lambda_j+1)$ is addable, in
which case its weight is obtained by adding a node to the $j$th row of
$\lambda$.
\end{rmk}

\begin{prop}\label{p:2}
  Suppose that $b$ is a maximal vector of weight $\lambda$ in some
  $\UU_q(\fgl_{n+1})$-module $M$. Assume that $\Phi_n(b)$ is maximal
  with respect to $\UU_q(\fgl_n)$. Then
  \[
  E_j \Phi_n^+(b) = 
  \begin{cases}
    (-q^{-1})^{n-1} \, v_1 \otimes \Psi^+_{n-1} b & \text{ if } j=1 \\
    0 & \text{ if } 1 < j \le n.
  \end{cases}
  \]
\end{prop}

\begin{proof}
As $\Phi_n(b) = \Phi^+_{n-1}(b) + (-q^{-1})^{n-1}\, v_1 \otimes
\Psi_{n-1} b$, it follows that
\[
\Phi^+_{n}(b) = \Phi^{++}_{n-1}(b) + (-q^{-1})^{n-1}\, v_2 \otimes
\Psi^+_{n-1} b .
\]
The result in case $j=1$ now follows by applying $E_1$ to both sides,
since $E_1$ acts as zero on all the terms of $\Phi^{++}_{n-1}(b)$ and
also acts as zero on $\Psi^+_{n-1} b$, since $E_1$ commutes with
$\Psi^+_{n-1}$.

By hypothesis, $E_1, \dots, E_{n-1}$ all act as zero on
$\Phi_n(b)$. Hence, $E_2, \dots, E_{n}$ all act as zero on
$\Phi^+_n(b)$. This proves the $j>1$ cases.
\end{proof}

The following will be used for the inductive step in the proof of
Theorem~\ref{t:main}.

\begin{thm}\label{t:main-step}
  Suppose that $b$ is a maximal vector of weight $\lambda$ in some
  $\UU_q(\fgl_{n+1})$-module $M$. Assume that $\Phi_n(b)$ is maximal
  with respect to $\UU_q(\fgl_n)$ and that $\alpha_n^\vee(\lambda) \ne
  0$. Then $\Phi_{n+1}(b)$ is maximal with respect to
  $\UU_q(\fgl_{n+1})$.
\end{thm}

\begin{proof}
From the recursive formula \eqref{e:Phi-recursion}, we have 
\[
\Phi_{n+1}(b) = \Phi^+_n(b) + (-q^{-1})^n v_1 \otimes \Psi_n(b).
\]
Since $E_1$ acts as zero on $v_1$, the definition of the coproduct
$\Delta$ and the first case in Propositions~\ref{p:1} and
\ref{p:2} gives
\begin{align*}
  E_1 \Phi_{n+1}(b) &= E_1 \Phi_n^+(b) + (-q^{-1})^n \tK_1 v_1 \otimes
  E_1 \Psi_n(b)\\ &= (-q^{-1})^{n-1} v_1 \otimes \Psi_{n-1}^+ b -
  (-q^{-1})^{n-1} v_1 \otimes \Psi_{n-1}^+ b = 0.
\end{align*}
The second case in the same Propositions ensures that $E_2, \dots,
E_n$ all act as zero on both terms of the right hand side of
$\Phi_{n+1}(b)$, finishing the proof.
\end{proof}

We now prove Theorem~\ref{t:main}.  Let $b$ be a maximal vector of
weight $\lambda$ in some $\UU_q(\fgl_{n+1})$-module $M$. Regard $M$ as
a $\UU_q(\fgl_n)$-module by restriction (using the obvious embedding
of $\UU_q(\fgl_n)$ in $\UU_q(\fgl_{n+1})$ given by $E_i \mapsto E_i$
and $F_i \mapsto F_i$ for $i=1,\dots, n-1$, and $K_j \mapsto K_j$ for
$j=1,\dots,n$). By the inductive hypothesis, we know that:
\begin{enumerate}
\item[1)] $\Phi_1(b)$ is maximal with respect to $\UU_q(\fgl_{n})$, and

\item[2)] for each $2\le m \le n$, $\Phi_{m}(b)$ is maximal with
  respect to $\UU_q(\fgl_{n})$, provided that
  $\alpha_{m-1}^\vee(\lambda) \ne 0$.
\end{enumerate}
It is clear that $\Phi_1(b)$ is maximal with respect to
$\UU_q(\fgl_{n+1})$. For each $m$ with $2 \le m \le n$, $\Phi_{m}(b)$
is maximal with respect to $\UU_q(\fgl_{n+1})$, under the stated
proviso in 2), since the shifted $\Psi_j$ appearing in the formulas
depend only on $F_1, \dots, F_{n-1}$ and thus $E_n$ acts as zero on
all the terms. Finally, the fact that $\Phi_{n+1}(b)$ is maximal with
respect to $\UU_q(\fgl_{n+1})$, provided that
$\alpha_{n}^\vee(\lambda) \ne 0$, is the content of
Theorem~\ref{t:main-step}. This completes the inductive step, and thus
the proof of Theorem~\ref{t:main}.

\begin{rmk}
An analysis of the above proof reveals the following. Assume that $b$
is maximal and that $\Phi_n(b) = \Phi^+_{n-1}(b) + (-q^{-1})^{n-1} v_1
\otimes \Psi_{n-1} b$ is maximal with $\Psi_{n-1} b$ defined as above
in terms of $c_{n-1}$ and $d_{n-1}$. Suppose that we define
\[
\Psi_n b = x F_1 \Psi^+_{n-1} - y  \Psi^+_{n-1} F_1
\]
with undetermined coefficients $x$, $y$ and set
\[
\Phi_{n+1}(b) = \Phi^+_n(b) + (-q^{-1})^n v_1 \otimes \Psi_n b.
\]
The two necessary conditions $E_1 \Phi_{n+1}(b) = 0$ and $E_2
\Phi_{n+1}(b) = 0$ are equivalent to a linear system of two equations
in the two unknowns $x$, $y$. Solving that linear system determines
that $x = [c_n]/[d_n]$, $y = [c_n -1]/[d_n]$ uniquely. The proof of
Lemma \ref{l:3} then shows that $E_j$ acts as zero on $\Phi_{n+1}(b)$,
for any $3 \le j \le n$. In this sense, the sequence of scalars used
in the definition of $\Psi_n$ boils down to solving a $2 \times 2$
system.
\end{rmk}

As an application, the results of this section give us a $q$-analogue
of Young's rule.

\begin{cor}\label{c:Young's-rule}
  Let $\lambda$ be a partition into at most $n$ parts. Then the tensor
  product $V_q \otimes V_q(\lambda)$ has the multiplicity-free
  $\UU_q(\fgl_n)$-module decomposition
  \[
  V_q \otimes V_q(\lambda) \cong \textstyle \bigoplus_{\mu \setminus
    \lambda \, = \, \square} V_q(\mu)
  \]
  where the sum on the right hand side is over the set of partitions $\mu$
  which differ from $\lambda$ by one box (occupying an addable node).
\end{cor}

\begin{proof}
Let $b$ be a highest weight vector generating $V_q(\lambda)$ (so $b$
is maximal). Suppose that $\alpha_j^\vee(\lambda) \ne
0$. Equivalently, $(j,\lambda_j+1)$ is an addable node in the shape
$\lambda$. Let $\mu$ be the shape obtained by adding a node in that
position, so that $\mu \setminus \lambda \, = \, \square$. Then
$\Phi_j(b)$ is a maximal vector in $V_q \otimes V_q(\lambda)$ of
weight $\mu$. We obtain such a maximal vector in $V_q \otimes
V_q(\lambda)$ for each addable node, so $V_q \otimes V_q(\lambda)$
contains the (direct) sum of the submodules generated by the maximal
vectors of the form $\Phi_j(b)$ for which $\alpha_j^\vee(\lambda) \ne
0$. Now we can finish by a dimension comparison, using the fact (since
$q$ is not a root of unity) that the characters of the $q$-Weyl
modules are given by Weyl's character formula, and hence their
dimensions are the same as in the classical case.
\end{proof}

Using the operators $\Phi_1, \dots, \Phi_n$, we will now construct a
non-zero maximal vector $\cc_\pi$ corresponding to each walk (see
Section~\ref{s:pre}) in the Bratteli diagram.

\begin{defn}\label{d:Upsilon}
Let $\pi \mapsto \cc_\pi$ be the map from walks on the Bratteli
diagram to maximal vectors in the tensor algebra of $V_q$, defined as
follows.  If the unique node in $\pi^{(j)} \setminus \pi^{(j-1)}$ is
in the $k$th row, we set $\Upsilon_j = \Phi_k$. Then $\cc_\pi$ is
given by
\[
\cc_\pi = \Upsilon_r \Upsilon_{r-1} \cdots \Upsilon_1(1)
\]
if the walk $\pi$ has length $r$.  It follows from
Theorem~\ref{t:main} that $\cc_\pi$ is maximal. 
\end{defn}

In Section \ref{s:orthog} we show that the maximal vectors indexed by
walks of length $r$ are pairwise orthogonal with respect to a natural
bilinear form and span the space of maximal vectors in $V_q^{\otimes
  r}$.

\begin{example}
We denote $v_{i_1} \otimes v_{i_2} \otimes
\cdots$ by $v_{i_1i_2 \cdots}$ as a convenient shorthand.
The unique length $1$ walk in the Bratteli diagram produces the
maximal vector $\Phi_1(1) = v_1$. The two length $2$ walks produce the
maximal vectors $\Phi_1(\Phi_1(1)) = v_{11}$ and $\Phi_2(\Phi_1(1)) =
v_{21}-q^{-1}v_{12}$. There are four length $3$ walks, producing the
maximal vector
\[
\Phi_1(\Phi_1(\Phi_1(1))) = v_{111}
\]
of weight $(3)$, the two maximal vectors 
\begin{align*}
\Phi_1(\Phi_2(\Phi_1(1))) &= v_{121} - q^{-1} v_{112} \\
  \Phi_2(\Phi_1(\Phi_1(1))) &= v_{211}
  - \frac{q^{-1}}{[2]}(q^{-1} v_{121} + v_{112})
\end{align*}
of weight $(2,1)$, and finally the maximal vector
\[
\Phi_3(\Phi_2(\Phi_1(1))) = v_{321} - q^{-1} v_{312} - q^{-1} v_{231}
    + q^{-2} v_{213} + q^{-2} v_{132} - q^{-3} v_{123}
\]
of weight $(1,1,1)$. 
\end{example}

\section{Orthogonality}\label{s:orthog}\noindent
We fix $n$ throughout this section. Let $V_q = V_q(1)$, the vector
representation of $\UU_q = \UU_q(\fgl_n)$ with its standard basis
$\{v_1, \dots, v_n\}$.  Let $\bil{-}{-}$ be the nondegenerate
symmetric bilinear form on $V_q$ given by $\bil{v_i}{v_j} =
\delta_{ij}$.  The basis $\{v_1, \dots, v_n\}$ is orthonormal with
respect to this form. Extend $\bil{-}{-}$ to a nondegenerate symmetric
bilinear form on $V_q^{\otimes r}$, denoted by the same symbols, by
defining
\begin{equation}
  \bil{v_{i_1} \otimes \cdots \otimes v_{i_r}}{v_{j_1} \otimes \cdots
    \otimes v_{j_r}} = \textstyle\prod_\ell \bil{v_{i_\ell}}{v_{j_\ell}}.
\end{equation}
It is clear that weight vectors of different weight are orthogonal
with respect to the form; that is, for weight vectors $b$ and $b'$, we
have $\bil{b}{b'} = 0$ unless $b$ and $b'$ have the same weight.

The following is the main result of this section. Note that part (b)
implies that $\cc_\pi \ne 0$, for any walk $\pi$. (See \ref{d:Upsilon}
for the definition of $\cc_\pi$.)

\begin{thm}\label{t:orthogonality}
Let $\pi$ and $\pi'$ be walks on the Bratteli diagram. Then:
\begin{enumerate}
\item $\bil{\cc_\pi}{\cc_{\pi'}} = 0$ whenever $\pi \ne
  \pi'$.
\item $\bil{\cc_\pi}{\cc_\pi} \ne 0$.
\end{enumerate}
\end{thm}

Before taking up the proof, we note the following immediate consequence. 

\begin{cor}\label{c:orthogonality}
The set $\{\cc_\pi \mid \pi \in \Walk(r)\}$ is an orthogonal basis for
the space of maximal vectors in $V_q^{\otimes r}$.
\end{cor}

\begin{proof}
This follows from the decomposition Corollary~\ref{c:SWD}(a), by a
dimension count. The listed vectors are non-isotropic by
\ref{t:orthogonality}(b), and thus are linearly independent by
\ref{t:orthogonality}(a).
\end{proof}

The proof of Theorem~\ref{t:orthogonality}, which occupies the rest of
this section, is based on the following adjointness property of the
bilinear form, which generalizes a similar property observed in
\cite{DG:orthog}*{Lemma~3.7}.

\begin{lem}[adjointness]\label{l:adjointness}
Suppose that $b$ and $b'$ are weight vectors in some $V_q^{\otimes k}$ of
respective weights $\lambda$ and $\lambda'$.  Then
\[
\bil{E_i b}{b'} = q^{\alpha_i^\vee(\lambda)+1} \bil{b}{F_i b'}\quad \text{and} \quad
\bil{b}{E_i b'} = q^{\alpha_i^\vee(\lambda')+1} \bil{F_i b}{b'}
\]
for any $i<n$. Since weight vectors of different weight are
orthogonal, both sides of the displayed equalities are zero unless
$\lambda' = \lambda + \alpha_i$ and $\lambda' = \lambda - \alpha_i$,
respectively.
\end{lem}

\begin{proof}
As the two displayed equalities in the first claim are equivalent (use
symmetry of the form and interchange $b$, $b'$) it suffices to prove
the first. Furthermore, it suffices to check it on simple tensors, so
we may assume that
\[
b = v_{j_1} \otimes \cdots \otimes v_{j_k} \quad\text{and}\quad b' =
v_{j'_1} \otimes \cdots \otimes v_{j'_k}
\]
where $\lambda' = \lambda+\alpha_i$.  Since the simple tensors form an
orthonormal basis of $V_q^{\otimes k}$, we have
\[
\bil{Ab}{b'}=\bil{b}{A^Tb'}
\]
for any linear operator $A$ on $V_q^{\otimes k}$.  In particular,
$\bil{E_ib}{b'}=\bil{b}{E_i^Tb'}$. Thus, we need to compute
$E_i^T b'$. Recall that $E_i$ and $F_i$ act on $V_q^{\otimes k}$ via iterated
comultiplication:
\begin{align*}
\Delta^{(k-1)}(E_i) &= \sum_{j=1}^k \tK_i^{\otimes(j-1)}\otimes E_i \otimes
1^{\otimes(k-j)} \\ \Delta^{(k-1)}(F_i) &= \sum_{j=1}^k
1^{\otimes(j-1)}\otimes F_i \otimes (\tK_i^{-1}) ^{\otimes(k-j)}
\end{align*}
where $1$ denotes the identity operator on $V_q$.  From the definitions,
we have $E_i^T=F_i$ and $\tK_i^T=\tK_i$ as operators on $V_q$. Also,
$(A\otimes B)^T=A^T\otimes B^T$ for operators $A,B$. Hence,
\begin{align*}
(\Delta^{(k-1)}(E_i))^T &= \sum_{j=1}^k \tK_i^{\otimes(j-1)}\otimes F_i
\otimes 1^{\otimes(k-j)} \\ &= \left(\sum_{j=1}^k1^{\otimes(j-1)}\otimes
F_i\tK_i^{-1} \otimes (\tK_i^{-1}) ^{\otimes(k-j)}\right)\tK_i^{\otimes k}
\end{align*}
as operators on $V_q^{\otimes k}$. Since $\tK_i^{\otimes k}(b') =
q^{\alpha_i^\vee (\lambda')}\,b'$ and $F_i\tK_i^{-1}(v_j) =
q^{-1}F_i(v_j)$ for all $j=1,\ldots n$, we have $E_i^Tb' =
q^{\alpha_j^{\vee}(\lambda')-1}F_i(b)$. The result now follows from
the equality $q^{\alpha_j^{\vee}(\lambda')-1} =
q^{\alpha_j^{\vee}(\lambda)+1}$.
\end{proof}

It is now necessary to explicitly keep track of the dependence of
$\Psi_j = \Psi_{j,\lambda}$ on $\lambda$.  For notational convenience,
we set $\alpha_{i,j} = \alpha_i + \cdots + \alpha_j$ for any $i \le
j$. Then $\alpha_i = \alpha_{i,i}$ and we have
\begin{equation}
  d_{j,\lambda} = \alpha_{1,j}^\vee(\lambda) + j-1 \quad \text{and} \quad
  c_{j,\lambda} = \alpha_{2,j}^\vee(\lambda) + j-1.
\end{equation}
In terms of this notation, $\Psi_j = \Psi_{j,\lambda}$ is defined
recursively by the equations
\[
\Psi_{1,\lambda} = \frac{F_1}{[d_{1,\lambda}]}, \quad \Psi_{j,\lambda}
= \frac{1}{[d_{j,\lambda}]} \big( [c_{j,\lambda}]
F_1\Psi_{j-1,\lambda} - [c_{j,\lambda}-1] \Psi_{j-1,\lambda} F_1
\big) 
\]
for all $j \ge 2$. The following technical result will soon be needed.

\begin{lem}\label{l:alpha1}
Suppose that $j \ge 1$. Then:
\begin{enumerate}
\item $d_{j,\lambda}^{++} = d_{j,\lambda-\alpha_1}^{++}$ and
  $c_{j,\lambda}^{++} = c_{j,\lambda-\alpha_1}^{++}$.
\item $\Psi_{j,\lambda}^{++} = \Psi_{j,\lambda-\alpha_1}^{++}$.
\item $d_{j,\lambda}^{+} +1 = d_{j,\lambda-\alpha_1}^{+}$ and
  $c_{j,\lambda}^{+} = c_{j,\lambda-\alpha_1}^{+}$.
\item $\Psi_{j,\lambda}^{+} = \frac{[d_{j,\lambda}^+ + 1]}{[d_{j,\lambda}^+]}
  \Psi_{j,\lambda-\alpha_1}^{+}$.
\end{enumerate}
\end{lem}

\begin{proof}
Part (a) follows from the definitions since $\alpha_i^\vee(\alpha_1) =
0$ for all $i > 2$.
The calculation for $c_{j,\lambda}^{++}$ is similar.

The calculation $\Psi_{1,\lambda}^{++} = F_3 /
[\alpha_3^\vee(\lambda)] = F_3 / [\alpha_3^\vee(\lambda-\alpha_1)] =
\Psi_{1,\lambda-\alpha_1}^{++}$ proves the base case of part (b).
The proof continues by induction on $j$. If $j \ge 2$ and
$\Psi_{j-1,\lambda}^{++} = \Psi_{j-1,\lambda-\alpha_1}^{++}$ then by
Definition~\ref{d:Psi} and the inductive hypothesis we have
\begin{align*}
  \Psi_{j,\lambda-\alpha_1}^{++} &= \frac{1}{[d_{j,\lambda-\alpha_1}^{++}]}
  \big( [c_{j,\lambda-\alpha_1}^{++}] F_3 \Psi_{j-1,\lambda-\alpha_1}^{++}
  -  [c_{j,\lambda-\alpha_1}^{++} - 1] \Psi_{j-1,\lambda-\alpha_1}^{++} F_3 \big) \\
  & = \frac{1}{[d_{j,\lambda}^{++}]}
  \big( [c_{j,\lambda}^{++}] F_3 \Psi_{j-1,\lambda}^{++}
  -  [c_{j,\lambda}^{++} - 1] \Psi_{j-1,\lambda}^{++} F_3 \big)
  = \Psi_{j,\lambda}^{++}
\end{align*}
This completes the proof of part (b). 

Part (c) is proved by direct calculations similar to those in the
proof of part (a).

Part (d) follows directly from parts (b) and (c). By
Definition~\ref{d:Psi} we have
\begin{align*}
  \Psi_{j,\lambda}^{+} &= \frac{1}{[d_{j,\lambda}^+]} \left(
      [c_{j,\lambda}^+] F_2 \Psi_{j-1,\lambda}^{++} - [c_{j,\lambda}^+ - 1]
      \Psi_{j-1,\lambda}^{++} F_2 \right) \\
      &= 
      \frac{1}{[d_{j,\lambda}^+]} \left(
      [c_{j,\lambda-\alpha_1}^+] F_2 \Psi_{j-1,\lambda-\alpha_1}^{++}
      - [c_{j,\lambda-\alpha_1}^+ - 1]
      \Psi_{j-1,\lambda-\alpha_1}^{++} F_2 \right).
\end{align*}
The formula in (d) now follows by inserting the factor
$[d_{j,\lambda}^+ + 1] / [d_{j,\lambda-\alpha_1}^+] = 1$ in the
right hand side of the above and then rearranging.
\end{proof}

Now we are ready to prove the following crucial result.

\begin{prop}\label{p:reduction}
Let $b$ and $b'$ be maximal vectors of the same weight $\lambda$,
where $\lambda$ is a partition of not more than $n$ parts. 
\begin{enumerate}
\item If $j \ge 1$ then $\bigbil{\Psi_{j,\lambda} b}{\Psi_{j,\lambda} b'}
  = \kappa_j \bigbil{\Psi_{j-1,\lambda}^+ b}{\Psi_{j-1,\lambda}^+ b'}$,
  where
  \[
  \kappa_j = 
  \begin{cases}
    q^{1 - \alpha_1^\vee(\lambda)} / [d_{1,\lambda}]
      & \text{ if } j = 1\\
    q^{-\alpha_1^\vee(\lambda)} [c_{j,\lambda}] / [d_{j,\lambda}]
      & \text{ if } j \ge 2.
  \end{cases}
  \]
  
\item If $j \ge 2$ then $\bigbil{\Psi_{j,\lambda} b}{\Psi_{j-1,\lambda}^+ F_1 b'}
  = 0$.

\item If $j \ge 2$ then $\bigbil{\Psi_{j,\lambda} b}{F_1 \Psi_{j-1,\lambda} b'}
  = q^{-\alpha_1^\vee(\lambda)}
  \bigbil{\Psi_{j-1,\lambda}^+ b}{\Psi_{j-1,\lambda}^+ b'}$.
\end{enumerate}
\end{prop}

\begin{proof}
The three parts are interdependent, and will be proved by an
interleaved induction. Applications of Lemma~\ref{l:contraction}
appear frequently, and will be referred to as contractions.  We first
outline the argument and then elaborate on the details. Let $A_j$,
$B_j$, $C_j$ be the equalities in parts (a), (b), (c) respectively,
each of which depends on $j$.

\medskip
\emph{Overview.}
\par STEP 1.
Note that statement $C_j$ is true for all $j
\ge 2$ by direct calculation, using the adjointness lemma
(Lemma~\ref{l:adjointness}). This proves part (c) of
the proposition.

STEP 2. Prove $A_1$ directly, again using adjointness.

STEP 3. To prove $B_2$, we need to show that $\smbil{\Psi_{2,\lambda}
  b}{\Psi_{1,\lambda}^+ F_1 b'} = 0$. Expanding $\Psi_{2,\lambda}$ and
$\Psi_{1,\lambda}^+$ (see Example~\ref{ex:1}) we see after clearing
denominators that $B_2$ is equivalent to the equality
\[
\bil{[c_{2,\lambda}] F_1F_2 b - [c_{2,\lambda}-1] F_2F_1 b}{F_2F_1 b'}
= 0.
\]
This equality is verified by two applications of adjointness.

STEP 4. Observe that $B_j$ and $C_j$ together immediately imply $A_j$,
for all $j \ge 2$. This follows by expanding the second term in
$\bil{\Psi_{j,\lambda} b}{\Psi_{j,\lambda} b'}$ using
Definition~\ref{d:Psi}.

STEP 5. At this point, we know that statements $A_1$ and $B_2$ are
true. Thus $A_2$ is also true since $C_2$ is true.  Finally, we claim
that $A_j$ implies $B_{j+1}$ for all $j \ge 2$.  Once this claim is
proved, we conclude by induction that $A_j$ and $B_j$ are true for all
$j \ge 2$. This completes the proof, once the further details have
been verified.

\medskip
\emph{Further details.}  \par\noindent STEP 1.
From Definition~\ref{d:Psi} we have
\begin{equation}\label{e:step1}
\begin{aligned}
\bigbil{\Psi_{j,\lambda} b}{F_1 \Psi_{j-1,\lambda}^+ b'} &=
\frac{[c_{j,\lambda}]}{[d_{j,\lambda}]} \bigbil{F_1 \Psi_{j-1,\lambda}^+ b}{F_1
  \Psi_{j-1,\lambda}^+ b'} \\ &
\hspace{1.5cm} - \frac{[c_{j,\lambda}-1]}{[d_{j,\lambda}]}
\bigbil{\Psi_{j-1,\lambda}^+ F_1 b}{F_1 \Psi_{j-1,\lambda}^+ b'}.
\end{aligned}
\end{equation}
Set $a_i = \alpha_i^\vee(\lambda)$. Now we apply
adjointness (Lemma~\ref{l:adjointness}) twice to get
\begin{align*}
  \bigbil{F_1 \Psi_{j-1,\lambda}^+ b}{F_1 \Psi_{j-1,\lambda}^+ b'} &=
  q^{-a_1} \bigbil{E_1 F_1 \Psi_{j-1,\lambda}^+ b}{\Psi_{j-1,\lambda}^+ b'}, \\
  \bigbil{\Psi_{j-1,\lambda}^+ F_1 b}{F_1 \Psi_{j-1,\lambda}^+ b'} &=
  q^{-a_1} \bigbil{E_1 \Psi_{j-1,\lambda}^+ F_1 b}{\Psi_{j-1,\lambda}^+ b'}.
\end{align*}
Since $E_1$ commutes past $\Psi_{j-1,\lambda}^+$ in the right hand
side of the second equality, we can by Lemma~\ref{l:contraction}
contract an $E_1F_1$ in each equality to get
\begin{align*}
  \bigbil{F_1 \Psi_{j-1,\lambda}^+ b}{F_1 \Psi_{j-1,\lambda}^+ b'} &=
  q^{-a_1} [a_1+1] \bigbil{\Psi_{j-1,\lambda}^+ b}{\Psi_{j-1,\lambda}^+ b'}, \\
  \bigbil{\Psi_{j-1,\lambda}^+ F_1 b}{F_1 \Psi_{j-1,\lambda}^+ b'} &=
  q^{-a_1} [a_1] \bigbil{\Psi_{j-1,\lambda}^+ b}{\Psi_{j-1,\lambda}^+ b'}.
\end{align*}
Putting these last two equalities back into the right hand side of
\eqref{e:step1} yields
\[
\bigbil{\Psi_{j,\lambda} b}{F_1 \Psi_{j-1,\lambda}^+ b'} =
\frac{q^{-a_1}}{[d_{j,\lambda}]} \big( [c_{j,\lambda}][a_1+1] -
     [c_{j,\lambda}-1][a_1] \big)
\bigbil{\Psi_{j-1,\lambda}^+ b}{\Psi_{j-1,\lambda}^+ b'} .
\]
After an application of the $q$-identity lemma
(Lemma~\ref{l:q-int-identity}) we get the equality in statement $C_j$,
since $a_1+c_{j,\lambda} = d_{j,\lambda}$.

STEP 2. Statement $A_1$ is checked by a similar application of
adjointness; we leave this calculation to the reader.

STEP 3. We prove statement $B_2$.  As previously mentioned, we only
need to show that
\[
[c_{2,\lambda}] \bil{F_1F_2 b}{F_2F_1 b'} - [c_{2,\lambda}-1]
\bil{F_2F_1 b}{F_2F_1 b'} = 0 .
\]
Apply adjointness to the left hand side to get
\[
[c_{2,\lambda}] q^{-a_1} \bil{F_2 b}{E_1F_2F_1 b'} - [c_{2,\lambda}-1]
q^{-a_2} \bil{F_1 b}{E_2F_2F_1 b'}.
\]
In the first term above, commute $E_1$ with $F_2$.  After contracting the
occurrences of $E_1F_1$ and $E_2F_2$ in the first and second terms,
respectively, this becomes
\[
q^{-a_1} [c_{2,\lambda}][a_1] \bil{F_2 b}{F_2 b'} - q^{-a_2}
[c_{2,\lambda}-1][a_2+1] \bil{F_1 b}{F_1 b'}.
\]
Now we apply adjointness and contract one more time to rewrite the
above in the form
\[
q^{-a_1} [c_{2,\lambda}][a_1] q^{1-a_2} [a_2] \bil{b}{b'} - q^{-a_2}
[c_{2,\lambda}-1][a_2+1] q^{1-a_1} [a_1] \bil{b}{b'}
\]
and since $c_{2,\lambda} = a_2+1$ and the powers of $q$ are the same,
this simplifies to zero, as required.

STEP 4 needs no further details.

STEP 5. It remains only to prove the claim that $A_j$ implies
$B_{j+1}$, for $j \ge 2$. This is the most delicate part of the
argument. By Definition~\ref{d:Psi} applied to $\Psi_{j+1,\lambda} b$,
we have
\begin{equation}\label{e:11}
\begin{aligned}
\bigbil{\Psi_{j+1,\lambda} b}{\Psi_{j,\lambda}^+ F_1 b'} &=
\frac{[c_{j+1,\lambda}]}{[d_{j+1,\lambda}]}
\bigbil{F_1 \Psi_{j,\lambda}^+ b}{\Psi_{j,\lambda}^+ F_1 b'} \\
&\hspace{2cm} - \frac{[c_{j+1,\lambda}]-1}{[d_{j+1,\lambda}]}
\bigbil{\Psi_{j,\lambda}^+ F_1 b}{\Psi_{j,\lambda}^+ F_1 b'}.
\end{aligned}
\end{equation}
We now compute the two pairings on the right hand side of equation
\eqref{e:11}.  We begin with the first, which by adjointness satisfies
\[
\bigbil{F_1 \Psi_{j,\lambda}^+ b}{\Psi_{j,\lambda}^+ F_1 b'} =
q^{-a_1} \bigbil{\Psi_{j,\lambda}^+ b}{E_1\Psi_{j,\lambda}^+ F_1 b'}.
\]
We may commute $E_1$ with $\Psi_{j,\lambda}^+$ and then apply
contraction to the term $E_1F_1$ to obtain the simplification
\[
\bigbil{F_1 \Psi_{j,\lambda}^+ b}{\Psi_{j,\lambda}^+ F_1 b'} =
q^{-a_1} [a_1] \bigbil{\Psi_{j,\lambda}^+ b}{\Psi_{j,\lambda}^+ b'}.
\]
Now we apply a shifted version of statement $A_j$ to the right hand
side above to obtain the result
\begin{equation}\label{e:12}
  \bigbil{F_1 \Psi_{j,\lambda}^+ b}{\Psi_{j,\lambda}^+ F_1 b'} =
  q^{-a_1-a_2} \frac{[a_1] [c_{j,\lambda}^+]}{[d_{j,\lambda}^+]}
  \bigbil{\Psi_{j-1,\lambda}^{++} b}{\Psi_{j-1,\lambda}^{++} b'}.
\end{equation}
To compute the second pairing on the right hand side of \eqref{e:11},
we first apply Lemma~\ref{l:alpha1}(d) to get the equality
\[
\bigbil{\Psi_{j,\lambda}^+ F_1 b}{\Psi_{j,\lambda}^+ F_1 b'} =
\frac{[d_{j,\lambda}^+ + 1]^2}{[d_{j,\lambda}^+]^2}
\bigbil{\Psi_{j,\lambda-\alpha_1}^+ F_1 b}{\Psi_{j,\lambda-\alpha_1}^+ F_1 b'}.
\]
The next step is rather subtle. Observe that $F_1 b$ and $F_1 b'$ are
maximal vectors (each of weight $\lambda - \alpha_1$) with respect to
the parabolic root system obtained by deleting the first node of the
Dynkin diagram. Thus, we may apply a shifted version of the equality
in statement $A_j$ to write the pairing on the right hand side of the
above as a multiple of
$\bigbil{\Psi_{j-1,\lambda-\alpha_1}^{++} F_1b}%
{\Psi_{j-1,\lambda-\alpha_1}^{++} F_1 b'}$. With this, the right hand
side of the above takes the form 
\[
\frac{q^{-a_2-1}[d_{j,\lambda}^+ + 1]^2[c_{j,\lambda-\alpha_1}^+]}
     {[d_{j,\lambda}^+]^2 [d_{j,\lambda-\alpha_1}^+]}
     \bigbil{\Psi_{j-1,\lambda-\alpha_1}^{++} F_1 b}
            {\Psi_{j-1,\lambda-\alpha_1}^{++} F_1 b'}.
\]
Now we commute the $F_1$ to the left of
$\Psi_{j-1,\lambda-\alpha_1}^{++}$ in each term of the pairing and
then apply adjointness to rewrite the above in the form
\[
\frac{q^{-a_2-1}[d_{j,\lambda}^+ + 1]^2[c_{j,\lambda-\alpha_1}^+]}
     {[d_{j,\lambda}^+]^2 [d_{j,\lambda-\alpha_1}^+]}
     q^{1-a_1} \bigbil{E_1F_1\Psi_{j-1,\lambda-\alpha_1}^{++} b}
            {\Psi_{j-1,\lambda-\alpha_1}^{++} b'}.
\]
Contracting the occurrence of $E_1F_1$ (and combining the
powers of $q$) yields the expression
\[
\frac{q^{-a_1-a_2}[d_{j,\lambda}^+ + 1]^2[c_{j,\lambda-\alpha_1}^+][a_1]}
{[d_{j,\lambda}^+]^2 [d_{j,\lambda-\alpha_1}^+]}
\bigbil{\Psi_{j-1,\lambda-\alpha_1}^{++} b}
{\Psi_{j-1,\lambda-\alpha_1}^{++} b'}.
\]
But $d_{j,\lambda}^+ + 1 = d_{j,\lambda-\alpha_1}^+$ and
$c_{j,\lambda}^+ = c_{j,\lambda-\alpha_1}^+$, so once again applying
Lemma~\ref{l:alpha1} the above takes the form
\[
\frac{q^{-a_1-a_2}[d_{j,\lambda}^+ + 1][c_{j,\lambda}^+][a_1]}
{[d_{j,\lambda}^+]^2}
\bigbil{\Psi_{j-1,\lambda}^{++} b}{\Psi_{j-1,\lambda}^{++} b'}.
\]   
Finally, we put this and the right hand side of \eqref{e:12} back into
the right hand side of equation \eqref{e:11}, to obtain the following
scalar 
\[
\frac{q^{-a_1-a_2}[a_1][c_{j,\lambda}^+]}{[d_{j,\lambda}^+]^2[d_{j+1,\lambda}]}
\left( [c_{j+1,\lambda}] [d_{j,\lambda}^+] -
     [c_{j+1,\lambda} - 1][d_{j,\lambda}^+ + 1] \right)
\]
multiplied by $\bigbil{\Psi_{j-1,\lambda}^{++}
  b}{\Psi_{j-1,\lambda}^{++} b'}$.  But $d_{j,\lambda}^+ =
c_{j+1,\lambda} - 1$, so the above scalar evaluates to zero, and thus
we conclude that the left hand side of \eqref{e:11} is equal to zero.
This is statement $B_{j+1}$, so the claim is proved.
\end{proof}

From now on, we will fix $\lambda$ and suppress the dependence on
$\lambda$ in the notation. The reduction formula in
Proposition~\ref{p:reduction}(a) gives the following.

\begin{thm}\label{t:Phi-form}
Suppose that $b$ and $b'$ are maximal vectors of weight
$\lambda$. Fix $\lambda$ and set $c_j = c_{j,\lambda}$ and $d_j =
d_{j,\lambda}$. Then for all $m \ge 1$ we have:
\begin{enumerate}
\item $\bigbil{\Psi_m b}{\Psi_m b'} = \tau_m \bigbil{b}{b'}$, where
  \[   \tau_m = 
  \begin{cases}
  q^{1-\alpha_1^\vee(\lambda)} \frac{1}{[d_{1}]} & \text{ if
    $m=1$} \\
  q^{1-\alpha_{1,m}^\vee(\lambda)}
  \frac{1}{[d_1^{+(m-1)}]} \frac{[c_2^{+(m-2)}]}{[d_2^{+(m-2)}]}
  \cdots \frac{[c_m]}{[d_m]} & \text{ if $m \ge 2$}.
  \end{cases}
  \]
\item $\bigbil{\Phi_m(b)}{\Phi_m(b')} = \rho_m \bigbil{b}{b'}$, where
  \[
  \rho_m = q^{1-m} \frac{[d_1^{+(m-2)}+1]}{[d_1^{+(m-2)}]}
  \frac{[d_2^{+(m-3)}+1]}{[d_2^{+(m-3)}]} \cdots
  \frac{[d_{m-1}+1]}{[d_{m-1}]}.
  \]
\end{enumerate}
\end{thm}

\begin{proof}
Part (a) is immediate from Proposition~\ref{p:reduction}(a). Part (b)
follows from part (a) and the sum formula
\[
\bigbil{\Phi_m(b)}{\Phi_m(b')} = \sum_{j=0}^{m-1} q^{-2j}
\bigbil{\Psi_j^{+(m-1-j)} b}{\Psi_j^{+(m-1-j)} b'}.
\]
Note that $\bigbil{\Psi_0\, b}{\Psi_0\, b'} = \bigbil{b}{b'}$, as
$\Psi_0=1$. In general, the result in (b) follows by induction on
$m$. For $m \ge 1$, assuming that $\rho_m$ and $\tau_m$ are given by
the stated formulas, we need to show that $\rho_{m+1} = \rho_m^+ +
q^{-2m} \tau_m$. To verify this, we begin with
\[
\rho_m^+ = q^{1-m} \frac{[d_1^{+(m-1)}+1]}{[d_1^{+(m-1)}]}
\frac{[d_2^{+(m-2)}+1]}{[d_2^{+(m-2)}]} \cdots
\frac{[d^+_{m-1}+1]}{[d^+_{m-1}]}.
\]
Now use the fact that $d_j^+ + 1 = c_{j+1}$, for all $j$, to rewrite
the above equality in the form
\[
\rho_m^+ = q^{1-m} \frac{[c_2^{+(m-2)}]}{[d_1^{+(m-1)}]}
\frac{[c_3^{+(m-3)}]}{[d_2^{+(m-2)}]} \cdots
\frac{[c_m]}{[d^+_{m-1}]}.
\]
Now add $q^{-2m} \tau_m$ to both sides, using the definition of
$\tau_m$. After factoring common terms in the result, we obtain
\[
\rho_m^+ + q^{-2m} \tau_m = q^{1-m} 
\frac{[c_2^{+(m-2)}] [c_3^{+(m-3)}] \cdots [c_m]}
     {[d_1^{+(m-1)}] [d_2^{+(m-2)}] \cdots [d_m]}
     \left( [d_m] + q^{-m - \alpha_{1,m}^\vee(\lambda)} \right).
\]
Since $d_m = \alpha_{1,m}^\vee(\lambda) + m-1$, we have $-d_m-1 = -m -
\alpha_{1,m}^\vee(\lambda)$.  By the second $q$-integer identity in
Lemma~\ref{l:q-int-identity}, the expression inside the parentheses in
the above displayed equality simplifies to $q^{-1} [d_m+1]$. Once
again using the equality $d_j^+ + 1 = c_{j+1}$, the above simplifies
to
\[
\rho_m^+ + q^{-2m} \tau_m = q^{-m}
\frac{[d_1^{+(m-2)} + 1] [d_2^{+(m-2)} + 1] \cdots [d_m+1]}
     {[d_1^{+(m-1)}] [d_2^{+(m-2)}] \cdots [d_m]}
\]
which is equal to $\rho_{m+1}$. The proof is complete.     
\end{proof}

We are finally ready to give the proof of
Theorem~\ref{t:orthogonality}.  The proof is by induction on the length
of the walks. Let $\pi$ and $\pi'$ be walks of the same length on the
Bratteli diagram. If they terminate at different nodes $\lambda \ne
\lambda'$ then $\bil{\cc_\pi}{\cc_{\pi'}} = 0$ because $\cc_\pi$ and
$\cc_{\pi'}$ have different weights (and weight vectors of distinct
weights are orthogonal). So we may assume that $\pi$ and $\pi'$ both
terminate at the same node $\lambda$.

If $\pi \ne \pi'$ then there must be two distinct intermediate nodes
$\mu \ne \mu'$ in the Bratteli diagram such that $\pi$ and $\pi'$
visit $\mu$ and $\mu'$, respectively.  Let $\pi_0$ and $\pi'_0$ be the
subwalks (of $\pi$ and $\pi'$) to these respective nodes. Then by
iterating Theorem~\ref{t:Phi-form} we know that
$\bil{\cc_\pi}{\cc_{\pi'}}$ may be written as a scalar multiple of
$\bil{b}{b'}$, where $b = \cc_{\pi_0}$ and $b' = \cc_{\pi'_0}$. Since
$\bil{b}{b'} = 0$ by the inductive hypothesis, we see that
$\bil{\cc_\pi}{\cc_{\pi'}} = 0$. This concludes the proof of part (a).

Part (b) follows from Theorem~\ref{t:Phi-form} and the fact that $q$
is not a root of unity. The proof of \ref{t:orthogonality} is complete.

\section{Further properties of the $\Psi$ operators}\label{s:further}
\noindent
In this section we fix $n$, and write $\UU_q = \UU_q(\fgl_n)$. Fix a
maximal vector $b$ (in some $\UU_q$-module) of weight $\lambda$.
The definition of $\Phi_m(b)$ (for $m = 1,\dots, n$) may be
written in the form
\[
\Phi_m(b) = v_m \otimes b \;\; + \;\;
\sum_{j=1}^{m-1} (-q^{-1})^j \, v_{m-j} \otimes \Psi_j^{+(m-1-j)} b.
\]
The summation on the right hand side above is vacuous in case $m=1$,
producing $\Phi_1(b) = v_1 \otimes b$. We wish to analyze the terms
appearing in that summation in the cases $m = 2, \dots, n$, where it
is not vacuous.

For $m \ge 2$, $\Phi_m(b)$ is defined if and only if
$\alpha_{m-1}^\vee(\lambda) \ne 0$.  By Lemma~\ref{l:wt-lem}, the
weight of $\Psi_j^{+(m-1-j)} b$ is $\lambda - \ep_{m-j} + \ep_m =
\lambda - (\alpha_{m-j}+ \cdots + \alpha_{m-1})$ and thus the operator
$\Psi_j^{+(m-1-j)}$ has weight $-(\alpha_{m-j}+ \cdots +
\alpha_{m-1})$.


We know that $b$ generates an isomorphic copy of the $q$-Weyl module
$V_q(\lambda)$. Assuming that $\alpha_{m-1}^\vee(\lambda) \ne 0$ if $m
\ge 2$, the formula for $\Phi_m(b)$ sets up a linear map
$\xi_m^\lambda: V_q(1) \to \UU_q^-$ defined on basis vectors by
\begin{equation*}
  \xi_m^\lambda : v_{m-j} \mapsto \Psi_j^{+(m-1-j)} \qquad \text{for
    $j = 0, \dots, m-1$}.
\end{equation*}
The map $\xi_m^\lambda$ depends on $\lambda$ but not on $b$.  In terms
of this notation, we have $\Phi_m(b) = v_m \otimes b +
\sum_{j=1}^{m-1} (-q^{-1})^j v_{m-j} \otimes \xi_m^\lambda(v_{m-j})
b$. Reindexing, this becomes
\[
\Phi_m(b) = v_m \otimes b \;\; + \;\;
\sum_{j=1}^{m-1} (-q^{-1})^{m-j} v_{j} \otimes \xi_m^\lambda(v_{j}) b
\]
where $\xi_m^\lambda(v_{j}) = \Psi_{m-j}^{+(j-1)}$ for $j = 1, \dots,
m-1$. We summarize these observations.

\begin{lem}
  For $m \ge 2$, assume that $\alpha_{m-1}^\vee(\lambda) \ne 0$. The
  map
  \[
  \xi_m^\lambda: V_q(1) \to \UU^-
  \]
  sends $v_j$ to $\Psi_{m-j}^{+(j-1)}$, an operator in $\UU^-$ of
  weight $-(\alpha_j + \cdots + \alpha_{m-1})$, for each $1 \le j \le
  m-1$.
\end{lem}

These negative root vectors form an interesting subset of vectors in
$\UU^-$.

\begin{lem}
 Fix $m \ge 2$. If $\alpha_{m-1}^\vee(\lambda) \ne 0$ then the map
 $\xi_m^\lambda: V_q(1) \to \UU_q^-$ is injective.
\end{lem}

\begin{proof}
By Lemma \ref{l:Psi-is-defined} and shifting, the condition
$\alpha_{m-1}^\vee(\lambda) \ne 0$ guarantees that each
$\xi_m^\lambda(v_{j})$ exists and is non-zero. The result then
follows from the fact that the weights of the $\xi_m^\lambda(v_{j})$
are distinct.
\end{proof}

Jimbo \cite{Jimbo} has given a recursive construction of root vectors
$\hat{E}_{ij}$ (for $1\le i<j \le n$) in $\UU_q^+$ and $\hat{F}_{ij}$
(for $1\le j < i \le n$) in $\UU_q^-$ by setting:
\begin{gather*}
  \hat{E}_{i,i+1} = E_i, \quad \hat{F}_{i+1,i} = F_i \\ \hat{E}_{ij} =
  \hat{E}_{ik} \hat{E}_{kj} - q \hat{E}_{kj} \hat{E}_{ik} \text{ if }
  i < k < j\\ \hat{F}_{ij} = \hat{F}_{ik} \hat{F}_{kj} - q^{-1}
  \hat{F}_{kj} \hat{F}_{ik} \text{ if } i > k > j.
\end{gather*}
Klimyk and Schm\"{u}dgen \cite{KS}*{\S7.3.1} work out explicit
commutation relations for these elements (using Jimbo's definition of
$\UU_q(\fgl_n)$, which differs slightly from ours).  A natural
question to ask is whether or not Jimbo's root vectors can be used to
construct a basis of $\UU$ analogous the Poincar\'{e}--Birkhoff--Witt
(PBW) basis of the enveloping algebra of $\fgl_n$. We do not know the
answer to this question.

Our recursive construction of the $\Psi$ operators produces negative
root vectors in many ways, all of which are different from Jimbo's
negative root vectors. They also produce positive root vectors in just
as many ways, because there is a unique automorphism $\omega$ on
$\UU$ that interchanges $E_i$ and$F_i$ for all $i
= 1, \dots, n-1$ and also interchanges $K_i$ and $K_i^{-1}$ 
for all $i = 1, \dots, n$.

\begin{thm}\label{t:root-vectors}
Let $\lambda$ be a partition into not more than $n$ parts.  Suppose that
$\alpha_{j}^\vee(\lambda) \ne 0$ for all $j = 1, \dots, n-1$.
\begin{enumerate}
\item The union of the sets
\[
\{ \xi_m^\lambda(v_{j}) \mid 1 \le j \le m-1 \}
\]
as $m$ runs from $2$ to $n$ is equal to a set of linearly independent
negative root vectors in $\UU^-$ (in bijection with the set
$\mathcal{R}^-$ of negative roots) and each such vector is a linear
combination of Coxeter monomials.

\item If $b$ is a maximal vector of weight $\lambda$, the union of the
  sets
\[
\{ \xi_m^\lambda(v_{j}) b \ne 0 \mid 1 \le j \le m-1 \}
\]
as $m$ runs from $2$ to $n$ is equal to a set of linearly independent
vectors in $V_q(\lambda)$, in bijection with a subset of $\{ \lambda -
\alpha \mid \alpha \in \mathcal{R}^+ \}$.
\end{enumerate}
\end{thm}

\begin{proof}
The first claim in part (a) follows from weight considerations. Linear
independence follows from the fact that the weights of the image
vectors are all distinct and non-zero (by
Lemma~\ref{l:Psi-is-defined}).  For the second claim in part (a),
combine Lemma~\ref{l:Coxeter} with Lemma~\ref{l:wt-lem}, in light of
Definition~\ref{d:Psi}.  Part (b) follows from part (a). Note that it
is necessary to collect the non-zero elements in the union.
\end{proof}

\begin{rmk}
(i) The weight $\lambda = (n-1, n-2, \dots, 1, 0)$ satisfies the
  hypothesis of the theorem, so we get negative root vectors in
  that case, satisfying the constraint $\alpha_j^\vee(\lambda) = 1$
  for all $j$.

(ii) By applying $\omega$ to a set of negative
  root vectors, we get a set of positive root vectors. One can ask
  under what conditions these root vectors can be used to
  construct PBW-type bases of $\UU_q^-$ and $\UU_q^+$, and thus also
  of $\UU_q$.

(iii) If $\lambda$ satisfies the hypothesis, part (b) of the theorem
  gives a set of linearly independent elements of $V_q(\lambda)$, which
  can of course be extended to a basis. 
\end{rmk}


\section{Comparison with the classical case}\label{s:q=1}\noindent
For the moment, we stick to our standing assumption, that $\Bbbk$ is 
a field containing a non-zero element $q$ which is not a root of unity. 
(Soon we will specialize $q$ to $1$.) Recall that, for any $\mu \in \X^\vee$,  
Lusztig \cite{Lusztig:90}*{2.19} defined elements 
$\left[\begin{smallmatrix} K_\mu; s\\t \end{smallmatrix}\right]$ 
in $\UU = \UU_q(\fgl_{n+1})$ by the rule
\[
    \sqbinom{K_\mu; s}{t} = 
\prod_{i=1}^t \frac{K_\mu q^{s-i+1} - K_\mu^{-1} q^{-(s-i+1)}}{q^i - q^{-i}}.
\]
We need only the special case $t=1$ in the above. We set 
$[K_\mu; s] = \left[\begin{smallmatrix} K_\mu; s\\1 \end{smallmatrix}\right]$, 
for notational convenience, so that
\[
[K_\mu; s] = \frac{K_\mu q^s - K_\mu^{-1} q^{-s}}{q - q^{-1}} .
\]
If $v$ is a weight vector of weight $\lambda$ then in terms 
of the scalars $c_n$, $d_n$ in Definition~\ref{d:Psi} we have:
\begin{equation}\label{e:c_n-action}
\begin{aligned}
    \relax [K_2K_{n+1}^{-1}; n-1] \, v &= [c_n]\, v 
    = [\lambda_2 - \lambda_{n+1} + n-1]\, v , \\
    [K_1K_{n+1}^{-1}; n-1] \, v &= [d_n]\, v 
= [\lambda_1 - \lambda_{n+1} + n-1]\, v
\end{aligned}
\end{equation}
This motivates the following.

\begin{defn}\label{d:newPsi}
Recursively define operators $\ov{\Psi}_n$ in $\UU = \UU_q(\fgl_{n+1}(\Bbbk))$ 
by setting $\ov{\Psi}_1 = F_1$ and (for all $n \ge 2$)
\[
\ov{\Psi}_n = 
F_1 \ov{\Psi}_{n-1}^+ [K_2K_{n+1}^{-1}; n-1] - \ov{\Psi}_{n-1}^+ F_1 
[K_2K_{n+1}^{-1}; n-2] . 
\]
\end{defn}

This definition was obtained from Definition ~\ref{d:Psi} by eliminating
denominators and replacing $[c_n]$ and $[c_n-1]$ by elements of $\UU^0$.
It is clear from \eqref{e:c_n-action} that the operators $\ov{\Psi}_n$ and
$\Psi^\lambda_n$ (see Definition~\ref{d:Psi}) act the same on weight vectors
of weight $\lambda$, up to scalar multiple. To be precise,
$\ov{\Psi}_n\, v = [d_n]\, \Phi^\lambda_n\, v$ if $v$ is a weight vector 
of weight $\lambda$.
Note that $\ov{\Psi}_n$ is independent of $\lambda$. 

From now on, we
let $\Bbbk$ be an arbitrary field of characteristic zero, and we take $q=1$ 
in $\Bbbk$. Regarding $\Bbbk$ as an $\A$-algebra via the specialization 
morphism sending $x \to 1$, we obtain an algebra $\UU_1 = \UU_1(\fgl_{n+1})$ 
over $\Bbbk$, defined by $\UU_1 = \UU_\A \otimes_\A \Bbbk$, as in 
\eqref{e:U_q}. By Lusztig \cite{Lusztig:90}*{8.16}, 
there is a surjective algebra morphism mapping $\UU_1$ onto
the universal enveloping algebra $\fU = \fU(\fgl_{n+1}(\Bbbk))$, with kernel
generated by all $K_j-1$, for $j = 1, \dots, n+1$. 
This is what we mean by the classical case.
As $[m]_{q=1} = m$ and $[K_j; s]_{q=1} = H_j + s$, 
the following is a classical analogue of 
Definition~\ref{d:newPsi}.

\begin{lem}\label{l:newPsiq1}
The $q=1$ version of the recursion given in Definition \upshape{\ref{d:newPsi}} 
is equivalent to setting
$\ov{\Psi}_1 = F_1$ and (for $n \ge 2$)
\[
\ov{\Psi}_n = 
F_1 \ov{\Psi}_{n-1}^+ (H_2-H_{n+1} + n-1) - \ov{\Psi}_{n-1}^+ F_1 
(H_2-H_{n+1} + n-2) .
\]
\end{lem}


For any $i,j$ we set $E_{i,j} := e_{i,j}$ (where the $e_{i,j}$ 
are the matrix units as in Section \ref{s:pre}) 
and $F_{i,j}:= E_{j,i}$. Let $C_{i,j}:= H_i-H_j+j-i$.
Regard these as elements of the universal enveloping algebra, as usual. 
Recall \cite{Carter} Carter's lowering operators $S_{i,j}$ 
(for $1 \le i < j \le n+1$) in $\fU(\fgl_{n+1}(\Bbbk))$, defined 
by the $(j-i) \times (j-i)$ determinant: 
\[
S_{i,j} = \left|
\begin{matrix}
    F_{i,i+1} & F_{i,i+2} & \dots & \dots & F_{i,j-1} & F_{i,j} \\
    -C_{i,i+1} & F_{i+1,i+2} & \cdots & \cdots & F_{i+1,j-1} & F_{i+1,j} \\
    0 & -C_{i,i+2} & F_{i+2,i+3} & \cdots & F_{i+2,j-1} & F_{i+2,j} \\
    \vdots & \vdots & \ddots & \ddots & \vdots & \vdots \\
    0 & 0 & \cdots & -C_{i,j-2} & F_{j-2,j-1} & F_{j-2,j} \\
    0 & 0 & \cdots & 0 & -C_{i,j-1} & F_{j-1,j}
\end{matrix}
\right|.
\]
The $(s,t)$-entry of the determinant is $F_{i+s-1,i+t}$ if $s\le t$,
$-C_{i+s-1,i+s}$ if $s=t+1$, and $0$ otherwise. We note that care is
required in interpreting this determinant, as $\fU(\fgl_{n+1}(\Bbbk))$
is a non-commutative ring. Here, we follow Brundan \cite{Brundan:98b} in
\emph{defining} the determinant of an $n \times n$ matrix $M=(M_{i,j})$ 
over a non-commutative ring to be equal to $\sum_{\pi \in \Sym_n} 
\text{sgn}(\pi)\, M_{1,1\pi} \cdots M_{n,n\pi}$. In other words,
every monomial in the expansion of the determinant involves a product 
in which the terms are taken in order of the rows of the matrix.

For our purposes, it turns out to be sufficient to consider only
the case of $S_{1,n+1}$ in $\fU(\fgl_{n+1}(\Bbbk))$ as $n$ varies.
(It is immediate from the definition that $S_{i,j} = S_{1,j+1-i}^{+(i-1)}$; 
that is, the $S_{i,j}$ for $i>1$ are obtained from the $S_{1,k}$ by shifting.)
The crucial observation for us is the following.

\begin{prop}\label{p:lowering}
Carter's lowering operator $S_{1,n}$ in $\fU(\fgl_{n}(\Bbbk))$ 
can be defined recursively for all $n \ge 2$ by $S_{1,2} = F_1$ 
and (for $n \ge 2$) 
\[
S_{1,n+1} = F_n S_{1,n}(H_1-H_n + n-1) - S_{1,n}F_n(H_1-H_n + n-2).
\]
\end{prop}

\begin{proof}[Proof sketch]
Expand (paying attention to the ordering convention) the determinant 
defining $S_{1,n+1}$ by a Laplace expansion on its last row. 
This gives the equality $S_{1,n+1} = A C_{1,n} + S_{1,n} F_{n,n+1}$,
where $A$ is a determinant with one fewer row and column. 
Now use the commutator formula $F_{i,n+1} = [F_n, F_{i,n}]$ 
to replace each term in the last column of $A$ by a difference 
of two terms. 
The result then follows by standard linearity properties of determinants. 
\end{proof}

Although the recursion in the last result looks similar to the
recursion in Lemma~\ref{l:newPsiq1}, there is an 
important difference, for which we now account. 
Recall that negative transpose ($X \mapsto -X^\tsp$) defines an 
automorphism of $\fgl_n(\Bbbk)$. This induces a corresponding 
automorphism of $\fU(\fgl_n(\Bbbk))$, defined on generators, by
\[
E_i \mapsto -F_i, \quad F_i \mapsto -E_i, \quad H_j \mapsto -H_j
\]
for all $i = 1, \dots, n-1$ and $j = 1, \dots, n$. When restricted 
to $\fU(\fsl_n(\Bbbk))$, it coincides with the well known automorphism 
in \cite{Humphreys:72}*{Prop.~14.3}. Presumably the next result is
well known; we include a proof because we were unable to find a 
suitable reference.

\begin{lem}\label{l:dia-aut}
There is an automorphism $\sigma_{n-1}$ of $\fU(\fgl_n(\Bbbk))$ 
defined on generators by 
\[
E_i \mapsto E_{n-i}, \quad F_i \mapsto F_{n-i},\quad 
H_j \mapsto -H_{n+1-j}
\]
for all $i = 1, \dots, n-1$ and $j = 1, \dots, n$. It 
corresponds to the Lie algebra automorphism obtained from the
negative transpose map $X \mapsto -X^\tsp$ by conjugating by the
$n \times n$ matrix $M = DJ$, where 
$D = \text{diag}(1,-1,1,-1,\cdots)$ and $J$ is the anti-diagonal 
matrix with all entries equal to $1$ on its anti-diagonal. 
\end{lem}

\begin{proof}
We need to see what happens when we conjugate the negative transpose 
map $X \mapsto -X^\tsp$ by $M = DJ$. 
We do this in two steps: first, we conjugate by $J$ and then we conjugate 
by the diagonal matrix $D$. Conjugation by $J$ flips indices 
$j \leftrightarrow n+1-j$, thus sends 
$E_i = E_{i,i+1}$ to $F_{n-i} = E_{n+1-i,n-i}$ and, similarly, 
sends $F_i = E_{i+1,i}$ to $E_{n-i} = E_{n-i,n+1-i}$. 
It also sends $H_i = E_{i,i}$ to $H_{n+1-i} = E_{n+1-i,n+1-i}$. 
Thus, by composing with the negative transpose map, 
we see that the map $X \mapsto -JX^{\mathsf{T}}J$ is given by
\[
E_i \mapsto -E_{n-i}, \quad F_i \mapsto -F_{n-i},\quad 
H_j \mapsto -H_{n+1-j}.
\]
The final step is to conjugate again by the diagonal matrix $D$, 
which changes the first two signs and preserves the third,
giving the result.
\end{proof}

\begin{rmk}
When restricted to $\fsl_n(\Bbbk)$,
the map $X \mapsto -DX^\tsp D^{-1}$ is given by
$E_i \mapsto E_{n-i}$, $F_i \mapsto F_{n-i}$, and 
$\alpha_i^\vee \mapsto \alpha_{n-i}^\vee$ 
(that is, $H_i - H_{i+1} \mapsto H_{n-i} - H_{n+1-i}$), for all 
$i = 1, \dots, n-1$. Thus, it is induced by
the graph automorphism that reverses the ordering of the nodes 
in the Dynkin diagram of type $\sf{A}$. It follows that, when restricted to 
$\fsl_n(\Bbbk)$, the map $X \mapsto -DX^\tsp D^{-1}$ is a realization 
of the corresponding outer automorphism.
(An outer automorphism of a semisimple Lie algebra 
is determined only up to conjugation; in type $\sf{A}$ 
the outer automorphism group has order 2.) 
\end{rmk}

We have two homomorphisms from $\fU(\fgl_n)$ into $\fU(\fgl_{n+1})$. 
One is given by 
the natural inclusion that maps $F_i$ to $F_i$, 
$E_i$ to $E_i$, and $H_j$ to $H_j$
for $1 \leq i \leq n-1$ and $1 \le j \le n$. Denote this map by
$\iota:\fU(\fgl_n) \to \fU(\fgl_{n+1})$. 
The other homomorphism is via our ``+'' map 
given by $F_i^+ = F_{i+1}$, $E_i^+ = E_{i+1}$, and $H_j^+ = H_{j+1}$   
for $1 \leq i \leq n-1$ and $1 \le j \le n$. 
These maps are intertwined by the family of $\sigma$ 
automorphisms described in Lemma~\ref{l:dia-aut}. 
More precisely, we have the following commutative diagram of morphisms
\begin{equation}\label{intertwines}
\begin{tikzcd}
\fU(\fgl_n) \arrow[r, "+"] \arrow[d, "\sigma_{n-1}"'] & \fU(\fgl_{n+1}) \arrow[d, "\sigma_{n}"] \\
\fU(\fgl_n) \arrow[r, "\iota"'] & \fU(\fgl_{n+1})
\end{tikzcd} 
\end{equation}
which will now be applied to compare the recursions for $\ov{\Psi}_n$ and $S_{1,n+1}$.

\begin{thm}\label{t:twist}
In the classical case, we have $\sigma_n(\ov{\Psi}_n) = S_{1,n+1}$,
for all $n$, as an equality in $\fU(\fgl_{n+1}(\Bbbk))$.
\end{thm}

\begin{proof}
As $C_{1,n} = H_1 -H_n + n-1$, we have
$C_{1,n}^+ = H_2 - H_{n+1} + n-1$. 
Thus, the recursion defining $\ov{\Psi}_n$ 
in Lemma~\ref{l:newPsiq1} is given by
\[
\ov{\Psi}_1 = F_1, \quad \ov{\Psi}_n = F_1 \ov{\Psi}_{n-1}^+ C_{1,n}^+  
- \ov{\Psi}_{n-1}^+ F_1 (C_{1,n}^+ - 1). 
\]
As $\sigma_1(\ov{\Psi}_1) = \sigma_1(F_1) = F_1 = S_{1,2}$ the 
result holds in the initial case. Assume by induction that
$\sigma_{n-1}(\ov{\Psi}_{j-1}) = S_{1,j}$, for some $j \ge 2$. Then
\[
\ov{\Psi}_j = F_1 \ov{\Psi}_{j-1}^+ C_{1,j}^+  
- \ov{\Psi}_{j-1}^+ F_1 (C_{1,j}^+ - 1). 
\]
It makes sense to apply the automorphism $\sigma_j$ to this 
equality in $\fU(\fgl_{j+1}(\Bbbk))$. The result is
\[
\sigma_j(\ov{\Psi}_j) = F_j \sigma_j(\ov{\Psi}_{j-1}^+) \sigma_j(C_{1,j}^+)  
- \sigma_j(\ov{\Psi}_{j-1}^+) F_j \sigma_j(C_{1,j}^+ - 1).
\]
The commutativity of the diagram \eqref{intertwines} shows that 
$\sigma_j(\ov{\Psi}_{j-1}^+) = \sigma_{j-1}(\ov{\Psi}_{j-1})$ and
$\sigma_j(C_{1,j}^+) = \sigma_{j-1}(C_{1,j})$. 
As $C_{1,j}$ is fixed by $\sigma_{j-1}$, it follows by 
the inductive hypothesis that
\[
\sigma_j(\ov{\Psi}_j) = F_j S_{1,j} C_{1,j} 
- S_{1,j} F_j (C_{1,j} - 1) = S_{1,j+1}
\]
where the last equality is by Proposition~\ref{p:lowering}. 
\end{proof}

\begin{rmk}
(i) Theorem \ref{t:twist} shows that, up to 
a twisting by the outer automorphism induced by the 
graph automorphism of the Dynkin diagram,
the $q=1$ version of the $\ov{\Psi}_n$-operator is the same as 
Carter's lowering operator $S_{1,n+1}$. 

(ii) In \cite{Brundan:98b}*{3.3}, Brundan defines a map 
``evaluation at $\lambda$'' on $\fU = \fU^- \fU^0 \fU^+$
(sending each PBW-basis element $FHE$ to $F\lambda(H)E$) 
in order to relate his generalized lowering operators to 
certain operators defined by Kleshchev \cite{Klesh}. 
Evaluation at $\lambda$ takes $(\ov{\Psi}_j)_{q=1}$ 
to a scalar multiple of $(\Psi_j)_{q=1}$.

(iii) Brundan \cites{Brundan:98a, Brundan:98b} 
defined a $q$-analogue of Carter's lowering operators, 
which depends on choosing root vectors in the quantized enveloping 
algebra. There is a $q$-analogue of Theorem~\ref{t:twist} in which
$\ov{\Psi}_n$ is the element in Definition \ref{d:newPsi} and 
$S_{1,n+1}$ is the corresponding $q$-lowering operator defined 
by Brundan, but the equality holds only up to a multiple 
by an element of $\UU^0$ that 
has image $1$ under the specialization map $\UU_1 \to \fU$ 
discussed above. The proof is similar to the proof of 
Theorem~\ref{t:twist}, but the calculations are more complicated.
For example, by expanding the appropriate determinant in 
Brundan's definition \cite{Brundan:98b}*{4.2}, we find that
\[
S_{1,3} = q^{-6} K_1^2K_2K_3 
\left(F_{2}F_{1} [K_1K_2^{-1};1] - F_{1}F_{2} [K_1K_2^{-1};0]\right).
\]
Under the specialization at $q=1$, this has image equal to the 
corresponding classical lowering operator defined by Carter. 
Comparing with 
\[
\ov{\Psi}_2 =  F_1F_2 [K_1K_2^{-1};1] 
- F_2F_1[K_1K_2^{-1};0]  
\]
we see an example of the claimed correspondence. 
%
\end{rmk}


\end{document}